\documentclass[12pt,reqno]{amsart}
\newcommand\ignore[1]{}
\input amssym.def
\usepackage[raiselinks=false,colorlinks=true,citecolor=blue,urlcolor=blue,linkcolor=blue,bookmarksopen=true,pdftex]{hyperref}
\usepackage{booktabs,multirow,lscape,longtable} 
\usepackage{mathtools} 
\usepackage{tikz} 
\numberwithin{equation}{section}
\numberwithin{equation}{subsection} 
\newtheorem{theorem}{Theorem}[section]

\newtheorem{lemma}[theorem]{Lemma}
\newtheorem{corollary}{Corollary}[theorem] 
\newtheorem{remark}[theorem]{Remark}

\begin{document} 
\baselineskip=15.5pt

\title
[Borel-de Siebenthal positive root systems]{Borel-de Siebenthal Positive Root Systems}  
\author{Pampa Paul }
\address{Department of Mathematics, Presidency University, 86/1 College Street, Kolkata 700073, India}
\email{pampa.maths@presiuniv.ac.in }
\subjclass[2020]{17B10, 17B20, 17B22, 17B25, 22E46.  \\ 
Keywords and phrases: Equi-rank Lie algebra, root system, positive root system, Dynkin diagram, discrete series representation.}

\thispagestyle{empty}
\date{}

\begin{abstract}
Let $G$ be a connected simple Lie group with finite centre, $K$ be a maximal compact subgroup of $G,$ and rank$(G)=$ rank$(K).$ 
Let $\frak{g}_0=$Lie$(G), \frak{k}_0=$Lie$(K) \subset \frak{g}_0, \frak{t}_0$ be a maximal abelian subalgebra of $\frak{k}_0, 
\frak{g}=\frak{g}_0^\mathbb{C}, \frak{k}=\frak{k}_0^\mathbb{C},$ and $\frak{h}=\frak{t}_0^\mathbb{C}.$ The existence of a 
Borel-de Siebenthal positive root system of $\Delta(\frak{g}, \frak{h})$ is proved by Borel and de Siebenthal. 
In this article, we have determined all Borel-de Siebenthal positive root systems of $\Delta(\frak{g}, \frak{h}),$ assuming the existence. 
As an application, we have determined 
the number of unitary equivalence classes of all Borel-de Siebenthal discrete series representations of $G$ (if $G/K$ is not Hermitian symmetric) 
with a fixed infinitesimal character.
\end{abstract}
\maketitle

\noindent 
\section{Introduction} 

Let $G$ be a connected simple Lie group with finite centre, $K$ be a maximal compact subgroup of $G,$ and rank$(G)=$ rank$(K).$ 
Let $\frak{g}_0$ be the Lie algebra of $G, \frak{k}_0$ be the subalgebra of $\frak{g}_0$ associated with the Lie subgroup $K$ of $G,$ and 
$\frak{t}_0$ be a maximal abelian subalgebra of $\frak{k}_0.$ Since rank$(G)=$ rank$(K), \frak{h}= \frak{t}_0^\mathbb{C}$ is a Cartan subalgebra of  
$\frak{g}=\frak{g}_0^\mathbb{C}$ as well as of $\frak{k}=\frak{k}_0^\mathbb{C}.$ Borel and de Siebenthal \cite{bds} have proved the existence of  
a positive root system of $\Delta = \Delta(\frak{g}, \frak{h})$ such that the associated set of simple roots contains exactly one non-compact simple 
root $\nu,$ and the coefficient $n_\nu(\delta)$ of $\nu$ in the highest root $\delta,$ when expressed as the sum of simple roots is $1,$ if $G/K$ is 
Hermitian symmetric, and $n_\nu (\delta)=2,$ if $G/K$ is not Hermitian symmetric. Borel-de Siebenthal positive root system has many applications 
in the representation theory of equi-rank Lie groups. For example, if $G/K$ is Hermitian symmetric, then a special positive root system which is used 
in \cite{hc1} to define holomorphic discrete series representations, is a Borel-de Siebenthal positive root system. If $G/K$ is not Hermitian symmetric, 
\O rsted and Wolf \cite{ow} have defined Borel-de Siebenthal discrete series representations of $G$ analogous to holomorphic discrete series representations, 
using a Borel-de Siebenthal positive root system. Also Borel-de Siebenthal positive root systems are used intrinsically to classify equi-rank non-compact 
real forms of a complex simple Lie algebra. See \cite[Ch. X]{helgason}, \cite[Ch. VI]{knappb}.  

Let $W_\frak{g}$ (respectively, $W_\frak{k}$) be the Weyl group of $\frak{g}$ (respectively, $\frak{k}$) with respect to the Cartan subalgebra $\frak{h}.$ 
If $P$ is a Borel-de Siebenthal positive root system of $\Delta,$ then so is $w(P)$ for all $w \in W_\frak{k}.$ Thus to determine the Borel-de 
Siebenthal positive root systems of $\Delta,$ it is sufficient to determine the Borel-de Siebenthal positive root systems containing a fixed positive 
root system $P_\frak{k}$ of $\Delta_\frak{k} = \Delta(\frak{k}, \frak{h}).$ Now to describe the main results of this article, we will introduce some more notations. 

Let $\frak{g}_0 = \frak{k}_0 \oplus \frak{p}_0$ be the Cartan decomposition associated to the maximal compact subgroup $K$ of $G,$ and 
$\frak{p}=\frak{p}_0^\mathbb{C}.$ Let $P$ be a Borel-de Siebenthal positive root system of $\Delta$ containing the fixed positive root system $P_\frak{k}$ of 
$\Delta_\frak{k}, \Phi$ be the set of all simple roots in $P,$ and $\nu \in \Phi$ be the unique non-compact root. 
Then we have a gradation of $\frak{g}$ with $\frak{p} = \frak{l}_{-1} \oplus \frak{l}_1,$ 
\[\frak{k} = 
\begin{cases}
\frak{l}_0  & \textrm{if } \frak{k} \textrm{ has non-zero centre}, \\
\frak{l}_{-2} \oplus \frak{l}_0 \oplus \frak{l}_2 & \textrm{if } \frak{k} \textrm{ is semisimple}; \\
\end{cases}
\] 
and $[\frak{l}_0, \frak{l}_i] \subset \frak{l}_i$ for all $i.$ 
The subalgebra $\frak{l}_0$ is a reductive, and $\frak{h}$ is a Cartan subalgebra of $\frak{l}_0.$ Let $\Delta_0 = \Delta(\frak{l}_0, \frak{h}).$ 
Then $P_0 = \Delta_0 \cap P_\frak{k}$ is a positive root system of $\Delta_0,$ and $\Phi_0 = \Phi \setminus \{\nu\}$ is the set of all simple roots in $P_0.$ Let 
$W_{\frak{l}_0}$ be the Weyl group of $\frak{l}_0$ relative to the Cartan subalgebra $\frak{h},$ and $w_{\frak{l}_0}^0 \in W_{\frak{l}_0}$ be the longest element 
with respect to the positive root system $P_0.$ Also the adjoint representation of $\frak{l}_0$ on $\frak{l}_i$ is irreducible for all $i \neq 0,$ and if $K$ is 
semisimple, the adjoint representation of $\frak{k}$ on $\frak{p}$ is irreducible. Note that the highest root $\delta$ of $\frak{g}$ is the highest weight of the 
irreducible $\frak{l}_0$-module $\frak{l}_2,$ and $\nu$ is the lowest weight of the irreducible $\frak{l}_0$-module $\frak{l}_1,$ with respect to the positive root system 
$P_0$ of $\Delta_0.$ Let $\lambda = w_{\frak{l}_0}^0(\nu),$ and $\epsilon =w_{\frak{l}_0}^0(\delta).$ Then $\lambda$ is the highest weight of the $\frak{l}_0$-module $\frak{l}_1$ 
relative to the positive root system $P_0$ of $\Delta_0,$ as well as of the $\frak{k}$-module $\frak{p}$ relative to the positive root system $P_\frak{k}$ of $\Delta_\frak{k}.$ 
Also $\Phi_\frak{k} = \Phi_0 \cup \{\epsilon\}$ is the set of all simple roots in $P_\frak{k}.$ If $\alpha \in \Delta,$ 
let $n_\phi (\alpha)$ denote the coefficient of $\phi$ in $\alpha$ when expressed as the sum of elements in $\Phi$ for all $\phi \in \Phi.$ 
Now we are ready to state the main results of this article. 

%
%
%
\begin{theorem}\label{main} 
Assume that $\frak{k}$ is semisimple. 

Let $\phi' \in \Phi_0$ be such that $n_{\phi'}(\delta)=1,$ and $\phi \in \Phi_0$ be such that $w_{\frak{l}_0}^0 (\phi') = -\phi.$ 
Let $\frak{l}'_0$ be the reductive subalgebra of $\frak{k}$ containing $\frak{h}$ and the Dynkin diagram of $[\frak{l}'_0, \frak{l}'_0]$ be 
the subdiagram of the Dynkin diagram of $\frak{k}$ with vertices $\Phi_\frak{k} \setminus \{\phi\}.$ Let $\nu'$ be the lowest weight of the irreducible $\frak{l}'_0$-submodule 
of $\frak{p}$ with highest weight $\lambda,$ relative to the positive roots system whose simple roots are $\Phi_\frak{k} \setminus \{\phi\}.$ 
Then $(\Phi_\frak{k} \setminus \{\phi \}) \cup \{\nu'\}$ is the set of all simple roots of a Borel-de Siebenthal positive root system of 
$\Delta$ containing $P_\frak{k}.$ 

Conversely, any Borel-de Siebenthal positive root system of $\Delta$ containing $P_\frak{k}$ is either $P$ or has the above form. 
\end{theorem}

\begin{corollary}\label{cor}
Let $\frak{g}$ be a complex simple Lie algebra, $\frak{g}_0$ be a non-compact real form of $\frak{g}$ with $\frak{k}_0,$ a maximal compactly imbedded subalgebra of $\frak{g}_0,$ 
rank$(\frak{g}_0)=$rank$(\frak{k}_0),$ and $\frak{k}_0$ be semisimple. Let $\frak{k}=\frak{k}_0^\mathbb{C}, \frak{h}$ be a Cartan subalgebra of $\frak{g},$ 
with $\frak{h}=\frak{t}_0^\mathbb{C}, \frak{t}_0 \subset \frak{k}_0,$ and $P_\frak{k}$ be a fixed positive root system of $\Delta_\frak{k}=\Delta(\frak{k}, \frak{h}).$ Then the 
number of Borel-de Siebenthal positive root systems of $\Delta = \Delta(\frak{g}, \frak{h})$ containing $P_\frak{k}$  is the covering index of the Lie group $Int(\frak{g}),$ the connected 
component of $Aut(\frak{g}).$ 
\end{corollary}

\begin{theorem}\label{rep}
Assume that $G/K$ is not a Hermitian symmetric space. Then the number unitary equivalence classes of Borel-de Siebenthal discrete series representations of $G$ with a fixed infinitesimal character \\
$=
\begin{cases}
1  & \textrm{if } \frak{g} = \frak{e}_8, \frak{f}_4, \frak{g}_2, \\
2  & \textrm{if } \frak{g} = \frak{b}_l(l \ge 2), \frak{c}_l (l \ge 2), \frak{e}_7, \\
3  & \textrm{if } \frak{g} = \frak{e}_6, \\
4  & \textrm{if } \frak{g} = \frak{\delta}_l (l \ge 4). \\
\end{cases}
$
\end{theorem}

Theorem \ref{main} is proved assuming the existence of a Borel-de Siebenthal positive root system of $\Delta.$ Here we have determined all Borel-de Siebenthal 
positive root system of $\Delta$ using the positive root system $P,$ by Theorem \ref{main}. 
These are precisely all Borel-de Siebenthal positive root systems described in Theorem \ref{main} and their $W_\frak{k}$-conjugate. Theorem \ref{rep} is the 
direct application of Theorem \ref{main}.

\noindent
\section{Borel-de Siebenthal positive root system}\label{bds}

Let $\frak{g}_0$ be a real simple Lie algebra, $\frak{g}_0=\frak{k}_0 \oplus \frak{p}_0$ be a Cartan decomposition with $\theta,$ the corresponding Cartan involution, and 
rank$(\frak{g}_0)=$rank$(\frak{k}_0).$ Let $\frak{t}_0$ be a maximal abelian subalgebra of $\frak{k}_0.$ Then $\frak{t}_0$ is a fundamental 
Cartan subalgebra of $\frak{g}_0.$ Let $\frak{g}=\frak{g}_0^\mathbb{C}, \frak{k}=\frak{k}_0^\mathbb{C}, \frak{p}=\frak{p}_0^\mathbb{C},$ and 
$\frak{h}=\frak{t}_0^\mathbb{C}.$ Then $\frak{h}$ is a Cartan subalgebra of $\frak{k}$ as well as of $\frak{g}.$ Let $\Delta = \Delta(\frak{g}, \frak{h})$ be the set of all non-zero 
roots of $\frak{g}$ relative to the Cartan subalgebra $\frak{h},$ and $\frak{g}^\alpha$ be the root space corresponding to a root $\alpha \in \Delta.$ Since $\frak{h} \subset \frak{k}, 
\frak{g}^\alpha$ is one-dimensional, and $[\frak{k}, \frak{k}] \subset \frak{k}, [\frak{k}, \frak{p}] \subset \frak{p},$ we have either $\frak{g}^\alpha \subset \frak{k},$ or 
$\frak{g}^\alpha \subset \frak{p}.$ A root $\alpha \in \Delta$ is said to be a compact root if $\frak{g}^\alpha \subset \frak{k},$ and 
$\alpha \in \Delta$ is said to be a non-compact root if $\frak{g}^\alpha \subset \frak{p}.$ Let $\Delta_\frak{k}$ denote the set of all compact roots in $\Delta,$ and 
$\Delta_n$ denote the set of all non-compact roots in $\Delta.$ Clearly $\frak{k}=\frak{h} +\sum_{\alpha \in \Delta_\frak{k}} \frak{g}^\alpha,
\frak{p}=\sum_{\beta \in \Delta_n} \frak{g}^\beta, \Delta_\frak{k}=\Delta(\frak{k}, \frak{h}),$ and $\Delta_n =\Delta \setminus \Delta_\frak{k}.$ 
We have the following theorem due to Borel-de Siebenthal. 

\begin{theorem}\cite{bds} \label{base}
There is a positive root system $P$ of $\Delta$ such that the corresponding set of simple roots $\Phi$ contains exactly one non-compact root, say $\nu,$ and $n_\nu(\delta),$ 
the coefficient of $\nu$ in the highest root $\delta$ when expressed as the sum of elements in $\Phi,$ is $1,$ if $\frak{k}$ has non-zero centre; and $n_\nu(\delta)$ is $2,$ if $\frak{k}$ 
is semisimple. 
\end{theorem}

A positive root system $P$ of $\Delta$ as in the Theorem \ref{base} is said to be a Borel-de Siebenthal positive root system. If $\delta$ is the highest root of $\frak{g}$ with respect to a 
Borel-de Siebenthal positive root system, then since $ [\frak{k}, \frak{p}] \subset \frak{p}, [\frak{p}, \frak{p}] \subset \frak{k},$ we have $\frak{g}^\delta \subset 
\frak{p},$ if $\frak{k}$ has non-zero centre; and $\frak{g}^\delta \subset \frak{k},$ if $\frak{k}$ is semisimple. 
Fix a positive root system $P_\frak{k}$ of $\Delta_\frak{k}.$ Let $W_\frak{g}$ be the Weyl group of $\frak{g}$ relative to the Cartan subalgebra $\frak{h},$  
$W_\frak{k}$ be the Weyl group of $\frak{k}$ relative to the Cartan subalgebra $\frak{h},$ and $w_\frak{k}^0 \in W_\frak{k}$ be the longest element with respect to 
the positive root system $P_\frak{k}.$ If $P$ is a Borel-de Siebenthal positive root system of $\Delta,$ then 
so is $wP$ for all $w \in W_\frak{k}.$ So we may assume that $P$ is a Borel-de Siebenthal positive root system of $\Delta$ containing $P_\frak{k}.$ 
Let $\Phi$ be the set of all simple roots in $P,$ $\nu \in \Phi$ be the unique non-compact root, $\delta$ be the highest root of $\frak{g}$ with respect to $P,$ 
and $\{\omega_\phi : \phi \in \Phi \}$ be the set of all fundamental weights corresponding to the set of simple roots $\Phi.$ 
For $i \in \mathbb{Z},$ define $\frak{l}_i = \{ X \in \frak{g} : [H_{\omega_\nu}, X] = iX \},$ where for any real valued linear function $\lambda$ on $i\frak{t}_0$ 
there exists unique $H_\lambda \in i\frak{t}_0$ such that 
$\lambda (H) = B (H, H_\lambda ) \textrm{ for all } H \in i\frak{t}_0, B$ denotes the Killing form of $\frak{g}.$ 
Put $\langle \lambda , \mu \rangle = B(H_\lambda, H_\mu)$ for any real valued linear functions $\lambda, \mu$ on $i\frak{t}_0.$
We have $[\frak{l}_0, \frak{l}_i] \subset \frak{l}_i$ for all $i \in \mathbb{Z},$ 
$\frak{p}= \frak{l}_{-1} \oplus \frak{l}_1,$ and 
\[\frak{k} = 
\begin{cases}
\frak{l}_0  & \textrm{if } \frak{k} \textrm{ has non-zero centre}, \\
\frak{l}_{-2} \oplus \frak{l}_0 \oplus \frak{l}_2 & \textrm{if } \frak{k} \textrm{ is semisimple}. \\
\end{cases}
\] 
Note that $\frak{l}_0$ is a reductive Lie subalgebra of $\frak{k},$ and $\frak{h}$ is a Cartan subalgebra of $\frak{l}_0.$ Let $\Delta_0 = \Delta(\frak{l}_0, \frak{h}).$ 
Then $P_0 = \Delta_0 \cap P_\frak{k}$ is a positive root system of $\Delta_0,$ and $\Phi_0 = \Phi \setminus \{\nu \}$ is the set of all simple roots in $P_0.$ Let 
$W_{\frak{l}_0}$ be the Weyl group of $\frak{l}_0$ relative to the Cartan subalgebra $\frak{h},$ and $w_{\frak{l}_0}^0 \in W_{\frak{l}_0}$ be the longest element 
with respect to the positive root system $P_0.$ Choose $X_\alpha (\neq 0) \in \frak{g}^\alpha,$ for all $\alpha \in \Delta.$ 

{\bf The adjoint representations of $\frak{l}_0$ on $\frak{l}_1$ and $\frak{l}_{-1}$ are irreducible:} 
Let $\beta \in \Delta$ be such that $\frak{g}^\beta \subset \frak{l}_1.$ Then $\beta$ can be written in the form $\beta = \nu +\phi_{i_1}+\phi_{i_2}+\cdots +\phi_{i_n}$ such that 
each partial sum from the left lies in $\Delta,$ where $\phi_{i_1}, \phi_{i_2}, \ldots , \phi_{i_n} \in \Phi_0,$ so that $ad(X_{\phi_{i_n}})\circ \cdots \circ ad(X_{\phi_{i_1}})(\frak{g}^\nu)=
\frak{g}^\beta.$ Thus $\frak{l}_1$ is an irreducible $\frak{l}_0$-module with lowest weight $\nu$ with respect to the positive root system $P_0.$ 
Similarly $\frak{l}_{-1}$ is an irreducible $\frak{l}_0$-module with highest weight $-\nu.$ 

{\bf If $\frak{k}$ is semisimple, the adjoint representations of $\frak{l}_0$ on $\frak{l}_2$ and $\frak{l}_{-2}$ are irreducible:} 
Since $\frak{g}$ is a simple complex Lie algebra with highest weight $\delta,$ any root $\alpha \in P,$ can be written as $\alpha = \delta -\phi_{j_1}-\phi_{j_2}-\cdots -\phi_{j_m}$ 
such that each partial sum from the left lies in $\Delta,$ where $\phi_{j_1}, \phi_{j_2}, \ldots , \phi_{j_m} \in \Phi,$ so that $ad(X_{-\phi_{j_m}})\circ \cdots \circ ad(X_{-\phi_{j_1}})
(\frak{g}^\delta)=\frak{g}^\alpha.$ Now $\frak{g}^\alpha \subset \frak{l}_2$ {\it iff} $\phi_{j_1}, \phi_{j_2}, \ldots , \phi_{j_m} \in \Phi_0.$
 Thus $\frak{l}_2$ is an irreducible $\frak{l}_0$-module with highest weight $\delta$ with respect to the positive root system $P_0.$ 
Similarly $\frak{l}_{-2}$ is an irreducible $\frak{l}_0$-module with lowest weight $-\delta.$ Let $\epsilon \in P$ be the lowest root such that $\frak{g}^\epsilon \subset \frak{l}_2.$ 
Then $\Phi_\frak{k}= \Phi_0 \cup \{\epsilon\}$ is the set of all simple roots in $P_\frak{k}.$ 

{\bf If $\frak{k}$ is semisimple, the adjoint representation of $\frak{k}$ on $\frak{p}$ is irreducible:} 
$\frak{p}= \frak{l}_{-1} \oplus \frak{l}_1,$ and if $\frak{k}$ is semisimple, $\frak{k}=\frak{l}_{-2} \oplus \frak{l}_0 \oplus \frak{l}_2.$ Since $\frak{k}$ is semisimple, and $\frak{p}$ is a 
finite dimensional $\frak{k}$-representation, $\frak{p}$ is completely reducible. Also $\frak{p}= \frak{l}_{-1} \oplus \frak{l}_1$ is the irreducible decomposition as an $\frak{l}_0$-module, 
$\frak{l}_0$ is a reductive subalgebra of $\frak{k}.$ Thus $\frak{p}$ is irreducible as $\frak{k}$-module {\it iff} $\frak{l}_1$ is not invariant under $\frak{k}.$ Let $\alpha \in \Delta$ be 
such that $\frak{g}^\alpha \subset \frak{l}_2.$ Then $\alpha = \phi_{k_1} + \cdots + \phi_{k_r} +\nu +\phi_{k_{r+1}}  +\cdots +\phi_{k_s}+\nu +\phi_{k_{s+1}} +\cdots +\phi_{k_t}$ 
such that each partial sum from the left lies in $\Delta,$ where $\phi_{k_1}, \ldots , \phi_{k_t} \in \Phi_0.$ Let $\beta = \phi_{k_1} + \cdots + \phi_{k_r} +\nu +\phi_{k_{r+1}} 
+\cdots +\phi_{k_s}.$ Then $\beta, \beta +\nu \in \Delta; \frak{g}^{-\beta} \subset \frak{l}_{-1}, \frak{g}^\nu \subset \frak{l}_1, \frak{g}^{-(\beta+\nu)} \subset \frak{l}_{-2}.$ 
Since $-\beta = -(\beta + \nu) + \nu,$ so $[\frak{l}_{-2}, \frak{l}_1] \neq 0,$ and clearly $[\frak{l}_{-2}, \frak{l}_1] \subset \frak{l}_{-1}$ which proves that $\frak{l}_1$ is not 
$\frak{k}$-invariant and consequently $\frak{p}$ is an irreducible $\frak{k}$-module. Let $\lambda \in P$ be the highest root such that $\frak{g}^\lambda \subset \frak{l}_1.$ Then 
$\lambda$ is the highest weight of the $\frak{k}$-module $\frak{p}$ (with respect to the positive root system $P_\frak{k}$) as well as of the $\frak{l}_0$-module $\frak{l}_1$ 
(with respect to the positive root system $P_0$). 

If $\alpha \in \Delta$ (respectively, $\alpha \in \Delta_\frak{k}$), let $n_\phi (\alpha)$ (respectively, $c_\phi (\alpha)$) denote the coefficient of $\phi$ in $\alpha$ when expressed 
as the sum of elements in $\Phi$ (repectively, $\Phi_\frak{k}$) for all $\phi \in \Phi$ (respectively, for all $\phi \in \Phi_\frak{k}$). If $\frak{k}$ is semisimple, let $C$ denote 
the component of the Dynkin diagram of $\frak{k}$ containing $\epsilon,$ and $\frak{k}_1$ denote the simple ideal of $\frak{k}$ whose Dynkin diagram is $C.$ 
Then $\delta$ is the highest root of $\frak{k}_1.$ 

\begin{lemma}\label{l1} 
Assume that $\frak{k}$ is semisimple. Let $\phi' \in \Phi_0$ be such that $n_{\phi'}(\delta) = 1,$ and 
$\phi \in \Phi_0$ be such that $w_{\frak{l}_0}^0 (\phi') = -\phi.$ Then $n_\phi(\lambda)=1, 
n_\phi(\epsilon)=1,$ and 
\[n_\phi(\delta) = 
\begin{cases}
1  & \textrm{if } \phi \notin C, \\
2 & \textrm{if } \phi \in C. \\
\end{cases}
\] 
\end{lemma}

\begin{proof} 
We have $\delta = 2\nu + \sum_{\psi \in \Phi_0} n_\psi (\delta) \psi,$ with $n_{\phi'}(\delta)=1.$ Let $w_{\frak{l}_0}^0 (\psi) = -\psi'$ for all $\psi \in \Phi_0.$ 
Since $w_{\frak{l}_0}^0$ has order $2, w_{\frak{l}_0}^0 (\phi') = -\phi$ implies $w_{\frak{l}_0}^0 (\phi) = -\phi',$ consistent with the above notation. 
Then $w_{\frak{l}_0}^0(\delta) = 2w_{\frak{l}_0}^0(\nu) - \sum_{\psi \in \Phi_0} n_\psi (\delta) \psi',$ that is $\epsilon = 2\lambda - \sum_{\psi \in \Phi_0} n_{\psi'}(\delta) \psi.$

So $2\lambda = \epsilon + \sum_{\psi \in \Phi_0} n_{\psi'}(\delta) \psi.$ Since $n_{\psi'}(\delta) > 0$ for all $\psi \in \Phi_0,$ we have $n_\psi(\lambda) > 0$ for all 
$\psi \in \Phi_0.$ Again $\lambda = \nu + \sum_{\psi \in \Phi_0} n_\psi(\lambda) \psi \implies w_{\frak{l}_0}^0(\lambda) = w_{\frak{l}_0}^0(\nu) - \sum_{\psi \in \Phi_0} n_\psi (\lambda) \psi'  
\implies \nu = \lambda - \sum_{\psi \in \Phi_0} n_{\psi}(\lambda) \psi'.$ So $\lambda = \nu + \sum_{\psi \in \Phi_0} n_{\psi}(\lambda) \psi'.$ Hence $n_\psi(\lambda) = n_{\psi'}(\lambda)$ 
for all $\psi \in \Phi_0.$ Since $0 \le n_{\phi'}(\lambda) \le n_{\phi'}(\delta) = 1,$ and $n_{\phi'}(\lambda) > 0,$ we have $n_\phi(\lambda)=n_{\phi'}(\lambda)=1.$ 

Now $\epsilon = 2\lambda - \sum_{\psi \in \Phi_0} n_{\psi'}(\delta) \psi,$ and $n_{\phi'}(\delta) = 1, n_\phi(\lambda) = 1.$ So $n_\phi (\epsilon) = 1,$ and 
$n_{\phi'}(\epsilon)= 2 - n_\phi (\delta).$ Since $n_{\phi'}(\epsilon) \ge 0,$ and $n_{\phi}(\delta) > 0,$ we have $n_{\phi}(\delta) = 1$ or $2.$

We have $\delta = \epsilon + \sum_{\psi \in \Phi_0} c_\psi (\delta) \psi,$ with $c_\phi (\delta) > 0$ {\it iff} $\phi \in C.$ Now $n_\phi(\delta) = 1$ or $2,$ and 
$n_\phi(\epsilon) = 1.$ Thus $c_\phi (\delta) = 0$ {\it iff} $\phi \notin C,$ and $c_\phi (\delta) =1$ {\it iff} $\phi \in C.$ Consequently, 
$n_\phi (\delta) = 1$ {\it iff} $\phi \notin C,$ and $n_\phi (\delta) =2$ {\it iff} $\phi \in C.$
\end{proof}

\begin{remark}\label{rl1} 
If $\frak{k}$ is semisimple, $\phi' \in \Phi_0$ is such that $n_{\phi'}(\delta)=1,$
$\phi \in \Phi_0$ is such that $w_{\frak{l}_0}^0 (\phi') = -\phi,$ and 
$\delta_2$ is the highest weight of the simple ideal of $\frak{k}$ whose Dynkin diagram is the component of the Dynkin diagram of $\frak{k}$ 
containing $\phi,$ then $c_\phi(\delta_2)=1.$ 
\end{remark}

\begin{lemma}\label{l2} 
Let $\frak{l}$ be a complex simple Lie algebra, $\frak{t}$ be a Cartan subalgebra, $\Delta^+$ be a positive root system of $\Delta(\frak{l}, \frak{t})$, and $V$ be a 
non-trivial irreducible $\frak{l}$-module with highest weight $\xi.$ Then the lowest weight of $V$ is given by 
$\xi - \sum_{ 1 \le i \le n} m_i \alpha_i, m_i \in \mathbb{N};$ where $\{\alpha_1, \alpha_2, \ldots , \alpha_n\}$ is the set of all simple roots in $\Delta^+.$ 
\end{lemma}

\begin{proof} 
Let $\langle . , . \rangle$ be an $Aut (\Delta(\frak{l}, \frak{t}))$ -invariant inner product on $\frak{t}_\mathbb{R}^* = \mathbb{R} \alpha_1 + \mathbb{R}\alpha_2 + 
\cdots + \mathbb{R}\alpha_n,$ and $\eta$ be the lowest weight of $V.$ Clearly $\eta = \xi - \sum_{ 1 \le i \le n} m_i \alpha_i,$ where 
$m_i \in \mathbb{N} \cup \{0\}$ for all $i.$ Let $S_1$ consists of all those simple roots $\alpha_i$ for which $m_i > 0,$ and 
$S_2$ consists of all those simple roots $\alpha_j$ for which $m_j = 0.$ Then the set of all simple roots $\{\alpha_1, \alpha_2, \ldots , \alpha_n\}$ is the 
disjoint union of $S_1$ and $S_2.$ Since $\xi$ is dominant weight, and $\xi \neq 0, \langle \xi, \alpha_{i_0} \rangle > 0,$ and so 
$\xi - \alpha_{i_0}$ is a weight of $V$ for some $1 \le i_0 \le n.$ Consequently $m_{i_0} > 0,$ and $\alpha_{i_0} \in S_1.$ So $S_1$ is non-empty. 
If $\alpha_j \in S_2,$ then $\langle \eta , \alpha_j \rangle = \langle \xi , \alpha_j \rangle - \sum_{1 \le i \le n, m_i > 0} m_i \langle \alpha_i , \alpha_j \rangle 
\ge 0,$ as $\langle \xi , \alpha_j \rangle \ge 0, \langle \alpha_i , \alpha_j \rangle \le 0$ for all $\alpha_i \in S_1.$ Since $\eta$ is a negative dominant weight, 
we have $\langle \eta , \alpha_j \rangle = 0,$ and so $\langle \alpha_i , \alpha_j \rangle = 0$ for all $\alpha_i \in S_1.$ Hence $S_1$ is orthogonal to $S_2,$ 
and $S_1$ is non-empty. Since $\frak{l}$ is simple, we have $S_1 = \{\alpha_1, \alpha_2, \ldots , \alpha_n\}.$ Consequently $S_2$ is empty and the lemma is proved. 
\end{proof}

Now we are ready to prove Theorem \ref{main}. 

\noindent
{\bf Proof of Theorem \ref{main}:} 
Let $P'$ be a Borel-de Siebenthal positive root system of $\Delta$ containing $P_\frak{k},$ $\Phi'$ be the set of all simple roots of $P',$ and $\nu' \in \Phi'$ be the unique 
non-compact root. Then we have a gradation $\frak{g}= \frak{l}'_{-2} \oplus \frak{l}'_{-1} \oplus \frak{l}'_0 \oplus \frak{l}'_1 \oplus \frak{l}_2$ with $\frak{k} = 
\frak{l}'_{-2} \oplus \frak{l}'_0 \oplus \frak{l}'_2, \frak{p} = \frak{l}'_{-1} \oplus \frak{l}'_1,$ and $[\frak{l}'_0, \frak{l}'_i] \subset \frak{l}'_i$ for all $i.$ 
The subalgebra $\frak{l}'_0$ is a reductive, and $\frak{h}$ is a Cartan subalgebra of $\frak{l}'_0.$ Let $\Delta'_0 = \Delta(\frak{l}'_0, \frak{h}).$ 
Then $P'_0 = \Delta'_0 \cap P_\frak{k}$ is a positive root system of $\Delta'_0,$ and $\Phi'_0 = \Phi' \setminus \{\nu' \}$ is the set of all simple roots in $P'_0.$ Let 
$W_{\frak{l}'_0}$ be the Weyl group of $\frak{l}'_0$ relative to the Cartan subalgebra $\frak{h},$ and $w_{\frak{l}'_0}^0 \in W_{\frak{l}'_0}$ be the longest element 
with respect to the positive root system $P'_0.$ Also the adjoint representation of $\frak{l}'_0$ on $\frak{l}'_1$ is irreducible, and $\frak{l}'_1$ is the 
$\frak{l}'_0$-submodule of $\frak{p}$ with highest weight $\lambda$ and lowest weight $\nu'$ with respect to the positive root system $P'_0.$ 
Let $\epsilon' \in P'$ be the lowest root such that $\frak{g}^{\epsilon'} \subset \frak{l}'_2.$ Then $\Phi'_0 \cup \{\epsilon'\}$ is the set of all simple roots in 
$P' \cap \Delta_\frak{k}=P_\frak{k}.$ So $\Phi'_0 \cup \{\epsilon'\}=\Phi_0 \cup \{\epsilon\},$ which implies either $\Phi'_0 = \Phi_0,$ or $\Phi'_0 = \Phi_\frak{k} 
\setminus \{\phi \}$ for some $\phi \in \Phi_0.$ Now $\Phi'_0 = \Phi_0 \implies \frak{l}'_0 = \frak{l}_0, \frak{l}'_1=\frak{l}_1, \nu'=\nu;$ so that 
$\Phi'=\Phi,$ and hence $P'=P.$ Assume that $P' \neq P$ so that $\Phi'_0 = \Phi_\frak{k} \setminus \{\phi \}$ for some $\phi \in \Phi_0,$ and $\epsilon'=\phi.$
It only remains to prove that if $w_{\frak{l}_0}^0 (\phi') = -\phi,$ then $n_{\phi'}(\delta)=1.$ Since $\phi \in P'$ with $\frak{g}^{\phi} \subset \frak{l}'_2,$ 
we have $\phi = 2 \nu' + m_\epsilon \epsilon + \sum_{\psi \in (\Phi_0 \setminus \{\phi\})} m_\psi \psi;$ $m_\epsilon, m_\psi \in \mathbb{N} \cup \{0\}$ for all 
$\psi \in \Phi_0 \setminus \{\phi\}.$ Also $w_{\frak{l}'_0}^0(\phi)$ is the highest weight of the $\frak{l}'_0$-module $\frak{l}'_2$ with respect to the positive 
root system $P'_0.$ So $w_{\frak{l}'_0}^0(\phi) = \phi + k_\epsilon \epsilon + \sum_{\psi \in (\Phi_0 \setminus \{\phi\})} k_\psi \psi;$ 
$k_\epsilon, k_\psi \in \mathbb{N} \cup \{0\}$ for all $\psi \in \Phi_0 \setminus \{\phi\}.$ This implies 
$w_{\frak{l}'_0}^0(\phi) = 2\nu' + (m_\epsilon+k_\epsilon) \epsilon + \sum_{\psi \in (\Phi_0 \setminus \{\phi\})} (m_\psi + k_\psi) \psi = 
2\nu' + d_\epsilon \epsilon + \sum_{\psi \in (\Phi \setminus \{\phi\})} d_\psi \psi,$ where $d_\epsilon = m_\epsilon + k_\epsilon, 
d_\psi = m_\psi + k_\psi$ for all $\psi \in \Phi_0 \setminus \{\phi\}.$ So $-w_{\frak{l}'_0}^0(\phi)=-2\nu' - d_\epsilon \epsilon 
- \sum_{\psi \in (\Phi_0 \setminus \{\phi\})} d_\psi \psi = -2\nu' + d'_\epsilon w_{\frak{l}'_0}^0(\epsilon) + \sum_{\psi \in (\Phi_0 \setminus \{\phi\})} 
d'_\psi w_{\frak{l}'_0}^0(\psi)$ for some $d'_\epsilon, d'_\psi \in \mathbb{N} \cup \{0\}$ for all $\psi \in \Phi_0 \setminus \{\phi\}.$ Thus 
$d'_\epsilon w_{\frak{l}'_0}^0(\epsilon) = 2\nu' - \sum_{\psi \in \Phi_0} d'_\psi w_{\frak{l}'_0}^0(\psi),$ where $d'_\phi=1.$ This implies 
$d'_\epsilon \epsilon = 2\lambda - \sum_{\psi \in \Phi_0} d'_\psi \psi.$ Since $n_\nu (\epsilon) = 2 = 2n_\nu(\lambda),$ we have $d'_\epsilon = 1.$ 
Hence $\epsilon = 2\lambda - \sum_{\psi \in \Phi_0} d'_\psi \psi \implies w_{\frak{l}_0}^0(\epsilon) = 2w_{\frak{l}'_0}^0(\lambda) + \sum_{\psi \in \Phi_0} d'_\psi \psi',$ 
where $w_{\frak{l}_0}^0 (\psi) = -\psi'$ for all $\psi \in \Phi_0.$ That is, $\delta = 2\nu + \sum_{\psi \in \Phi_0} d'_\psi \psi',$ so that $n_{\phi'}(\delta)=d'_\phi = 1.$

Conversely assume that $\phi' \in \Phi_0$ be such that $n_{\phi'}(\delta)=1,$
$\phi \in \Phi_0$ be such that $w_{\frak{l}_0}^0 (\phi') = -\phi,$ 
$\frak{l}'_0$ be the reductive subalgebra of $\frak{k}$ containing $\frak{h}$ and the Dynkin diagram of $[\frak{l}'_0, \frak{l}'_0]$ be 
the subdiagram of the Dynkin diagram of $\frak{k}$ with vertices $\Phi'_0 = \Phi_\frak{k} \setminus \{\phi\},$ and $\nu'$ be the lowest weight of the irreducible $\frak{l}'_0$-submodule 
$W$ of $\frak{p}$ with highest weight $\lambda.$ Let $P'_0$ be the positive root system associated with the simple system $\Phi'_0.$ Then $\Delta(\frak{l}'_0, \frak{h}) = 
\Delta([\frak{l}'_0, \frak{l}'_0],[\frak{l}'_0, \frak{l}'_0] \cap \frak{h})$ is given by $\Delta'_0 = P'_0 \cup (-P'_0).$ Let $\frak{l}'_\epsilon$ be the simple ideal of $\frak{l}'_0$ 
associated with the connected component of the Dynkin diagram of $\frak{l}'_0$ containing $\epsilon.$ Then $\frak{l}'_\epsilon$ is a subalgebra of $\frak{k}_1$ and 
$\frak{l}'_\epsilon = \frak{k}_1$ {\it iff} $\phi \notin C.$ Since $\langle \nu , \epsilon \rangle \neq 0$ and if $\phi \in C,$ $\langle \lambda , \psi \rangle \neq 0$ 
for some $\psi \in C \setminus \{\phi\}$ (see \S \ref{case}), 
we have the irreducible $\frak{l}'_\epsilon$-submodule of $W$ with highest weight $\lambda$ is non-trivial. 
So $\nu' = \lambda - \sum_{\psi \in \Phi'_0} c_\psi \psi,$ with $c_\epsilon > 0,$ by Lemma \ref{l2}. 
Since $P$ is Borel-de Siebenthal, $n_\nu(\epsilon) = 2, n_\nu (\lambda) = 1$ and $\nu' \in \Delta,$ we have $c_\epsilon = 1.$ Hence $\nu' \in \Delta_n \cap (-P)$ 
with $n_\phi (\nu') = 0,$ by Lemma \ref{l1}. Now we have $n_\nu(\epsilon)=2, n_\phi(\epsilon)=1, n_\nu (\nu')=-1,$ and $n_\phi(\nu')=0;$ which imply that 
$\Phi'_0 \cup \{\nu'\}$ is a basis of the dual space $\frak{h}^*.$ For $\alpha \in \Delta,$ let $m_\psi(\alpha)$ denote the coefficient of $\psi$ in $\alpha$ when expressed 
as the linear combination of elements in $\Phi'_0 \cup \{\nu'\}$ for all $\psi \in \Phi'_0 \cup \{\nu'\}.$ If $\beta \in \Delta_n$ with $n_\phi (\beta) = 1,$ then clearly 
$\beta$ is a weight of the module $W,$ and thus $m_{\nu'}(\beta)=1, m_\epsilon (\beta)= 1, m_\psi (\beta) \in \mathbb{N} \cup \{0\}$ for all $\psi \in \Phi'_0 \setminus \{\epsilon\}.$ 
Since $n_\phi (\lambda) = 1,$ $\nu, \nu'$ are the only minimal elements of the set $\{ \beta \in \Delta_n : n_\phi ( \beta) = 0\}$ of $\frak{k}$-weights of $\frak{p}.$ 
Thus $\nu'$ is the lowest element of $\{ \beta \in \Delta_n \cap (-P) : n_\phi (\beta)=0\}.$ Hence if $\beta \in \Delta_n \cap (-P)$ 
with $n_\phi (\beta)=0,$ then $m_{\nu'}(\beta) = 1, m_{\epsilon}(\beta) = 0, m_\psi (\beta) \in \mathbb{N} \cup \{0\}$ for all $\psi \in \Phi'_0 \setminus \{\epsilon\}.$ 
Thus we have if $\beta \in \Delta_n,$ then for all $\psi \in \Phi'_0 \cup \{\nu'\},\ m_\psi (\beta) \in \mathbb{Z}$ and are of same sign. Also $m_{\nu'}(\beta) = \pm 1.$ 
If $\alpha \in P_\frak{k},$ then $c_\phi (\alpha) = 0$ or $1,$ by Remark \ref{rl1}. If $c_\phi (\alpha) = 0,$ then clearly $m_{\nu'} (\alpha) = 0,$ and 
$m_\psi (\alpha) \in \mathbb{N} \cup \{0\}$ for all $\psi \in \Phi'_0.$ Since $n_\phi (\nu') = 0, n_\phi (\epsilon) = 1; 
\phi = 2\nu' + \epsilon + \sum_{\psi \in (\Phi'_0 \setminus \{\epsilon \})} m_\psi (\phi) \psi.$ 
We have $m_\psi (\phi) \in \mathbb{N} \cup \{0\}$ for all 
$\psi \in \Phi'_0 \setminus \{\epsilon \}.$ See \S \ref{case}. 
Then $m_\psi (\alpha) \in \mathbb{N} \cup \{0\}$ for all $\psi \in \Phi'_0$ and $m_{\nu'} (\alpha)=2,$ if 
$c_\phi (\alpha) = 1.$ Consequently $P'= \Phi'_0 \cup \{\nu'\}$ is a Borel-de Siebenthal positive root system of $\Delta$ containing $P_\frak{k}$ and 
the proof is complete.

\noindent 
\subsection{Case by case consideration:}\label{case}
1. $\frak{g}=\frak{b}_l (l \ge 2), \frak{g}_0 = \frak{so}(2p, 2l-2p+1), 2 \le p \le l:  \Phi = \{\phi_1, \phi_2, \ldots , \phi_l\}$ with $\nu = \phi_p,$ and $\delta = 
\phi_1+2\phi_2 + \cdots +2\phi_l.$ 

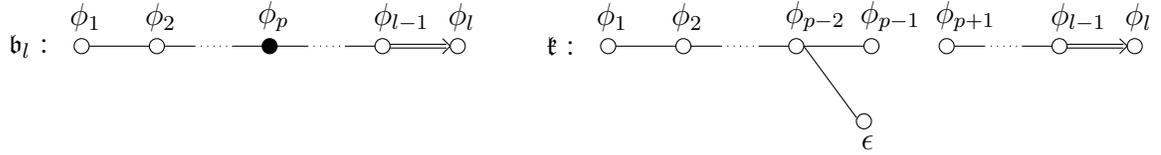
\begin{tikzpicture}

\draw  (0,0) circle [radius = 0.1]; 
\draw (1,0) circle [radius = 0.1]; 
\filldraw[black] (2.5,0) circle [radius = 0.1]; 
\draw (4,0) circle [radius = 0.1]; 
\draw (5,0) circle [radius = 0.1]; 
\node [above] at (0.05,0.05) {$\phi_1$}; 
\node [above] at (1.05,0.05) {$\phi_2$}; 
\node [above] at (2.55,0.05) {$\phi_p$}; 
 \node [above] at (4.25,0.05) {$\phi_{l-1}$}; 
\node [above] at (5.05,0.05) {$\phi_l$}; 
\node [left] at (-0.3,0) {$\frak{b}_l :$}; 

\draw (0.1,0) -- (0.9,0); 
\draw (1.1,0) -- (1.5,0);
\draw [dotted] (1.5,0) -- (2,0);
\draw (2,0) -- (2.4,0); 
\draw (2.6,0) -- (3,0); 
\draw [dotted] (3,0) -- (3.5,0); 
\draw (3.5,0) -- (3.9,0); 
\draw (4.9,0) -- (4.8,0.1); 
\draw (4.9,0) -- (4.8,-0.1); 
\draw (4.1,0.025) -- (4.85,0.025); 
\draw (4.1,-0.025) -- (4.85,-0.025); 

\draw (7,0) circle [radius = 0.1]; 
\draw (8,0) circle [radius = 0.1]; 
\draw (9.5,0) circle [radius = 0.1]; 
\draw (10.5,0) circle [radius = 0.1]; 
\draw (10.4,-1) circle [radius = 0.1]; 
\draw (11.5,0) circle [radius = 0.1]; 
\draw (13,0) circle [radius = 0.1]; 
\draw (14,0) circle [radius = 0.1]; 

\node [above] at (7.05,0.05) {$\phi_1$}; 
\node [above] at (8.05,0.05) {$\phi_2$}; 
\node [above] at (9.75,0.05) {$\phi_{p-2}$}; 
\node [above] at (10.75,0.05) {$\phi_{p-1}$}; 
\node [below] at (10.45,-1.05) {$\epsilon$}; 
\node [above] at (11.75,0.05) {$\phi_{p+1}$}; 
\node [above] at (13.25,0.05) {$\phi_{l-1}$}; 
\node [above] at (14.05,0.05) {$\phi_l$}; 
\node [left] at (6.7,0) {$\frak{k} :$}; 

\draw (7.1,0) -- (7.9,0); 
\draw (8.1,0) -- (8.5,0); 
\draw [dotted] (8.5,0) -- (9,0); 
\draw (9,0) -- (9.4,0); 
\draw (9.6,0) -- (10.4,0); 
\draw (9.6,0) -- (10.3,-0.95);
\draw (11.6,0) -- (12,0);
\draw [dotted] (12,0) -- (12.5,0); 
\draw (12.5,0) -- (12.9,0); 
\draw (13.9,0) -- (13.8,0.1); 
\draw (13.9,0) -- (13.8,-0.1); 
\draw (13.1,0.025) -- (13.85,0.025); 
\draw (13.1,-0.025) -- (13.85,-0.025); 

\end{tikzpicture} 

Here $P_\frak{k} = \{ \phi_i + \cdots +\phi_{j-1} : 1 \le i < j \le p\} \cup \{\phi_i + \cdots +\phi_{j-1} , \phi_i + \cdots +\phi_{j-1}+2\phi_j + \cdots +2\phi_l : p+1 \le i < j \le l \} 
\cup \{ \phi_i + \cdots +\phi_l : p+1 \le i \le l\} \cup \{\phi_i + \cdots +\phi_{j-1}+2\phi_j + \cdots +2\phi_l : 1 \le i < j \le p \},$ and $P \cap \Delta_n = 
\{ \phi_i + \cdots +\phi_l : 1 \le i \le p\} \cup \{\phi_i + \cdots +\phi_{j-1}, \phi_i + \cdots +\phi_{j-1}+2\phi_j + \cdots +2\phi_l : 1 \le i  \le p < j \le l \}.$ So $\epsilon = 
\phi_{p-1}+2\phi_p + \cdots +2\phi_l, \lambda = \phi_1 + \cdots +\phi_p+2\phi_{p+1} + \cdots +2\phi_l.$ We have $\langle \nu , \epsilon \rangle \neq 0$ and 
$\langle \lambda , \phi_1 \rangle \neq 0.$ The only possibility of $\phi' \in \Phi_0$ with $n_{\phi'} (\delta)=1$ is $\phi'=\phi_1.$ 

(i) Let $\phi'=\phi_1.$ Then $\phi = - w_{\frak{l}_0}^0 (\phi') = \phi_{p-1},$ and $\nu' = -(\phi_p +2\phi_{p+1}+\cdots +2\phi_l).$ So $\phi_{p-1} = \epsilon + 2\nu' + 
2\phi_{p+1} + \cdots + 2\phi_l.$ 

2. $\frak{g}=\frak{c}_l (l \ge 3), \frak{g}_0 = \frak{sp}(p, l-p), 1 \le p \le l-1:  \Phi = \{\phi_1, \phi_2, \ldots , \phi_l\}$ with $\nu = \phi_p,$ and $\delta= 
2\phi_1+\cdots+2\phi_{l-1}+\phi_l.$  

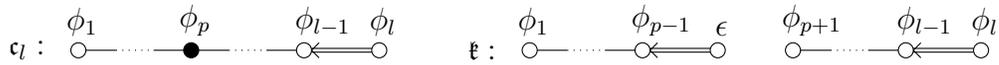
\begin{tikzpicture}

\draw (0,0) circle [radius = 0.1]; 
\filldraw[black] (1.5,0) circle [radius = 0.1]; 
\draw (3,0) circle [radius = 0.1]; 
\draw (4,0) circle [radius = 0.1]; 
\node [above] at (0.05,0.05) {$\phi_1$}; 
\node [above] at (1.55,0.05) {$\phi_p$}; 
\node [above] at (3.25,0.05) {$\phi_{l-1}$}; 
\node [above] at (4.05,0.05) {$\phi_l$}; 
\node [left] at (-0.3,0) {$\frak{c}_l :$}; 

\draw (0.1,0) -- (0.5,0);
\draw [dotted] (0.5,0) -- (1,0);
\draw (1,0) -- (1.4,0); 
\draw (1.6,0) -- (2,0); 
\draw [dotted] (2,0) -- (2.5,0); 
\draw (2.5,0) -- (2.9,0); 
\draw (3.1,0) -- (3.2,0.1); 
\draw (3.1,0) -- (3.2,-0.1); 
\draw (3.1,0.025) -- (3.9,0.025); 
\draw (3.1,-0.025) -- (3.9,-0.025); 

\draw (6,0) circle [radius = 0.1]; 
\draw (7.5,0) circle [radius = 0.1]; 
\draw (8.5,0) circle [radius = 0.1]; 
\draw (9.5,0) circle [radius = 0.1]; 
\draw (11,0) circle [radius = 0.1]; 
\draw (12,0) circle [radius = 0.1]; 

\node [above] at (6.05,0.05) {$\phi_1$}; 
\node [above] at (7.75,0.05) {$\phi_{p-1}$}; 
\node [above] at (8.55,0.05) {$\epsilon$}; 
\node [above] at (9.75,0.05) {$\phi_{p+1}$}; 
\node [above] at (11.25,0.05) {$\phi_{l-1}$}; 
\node [above] at (12.05,0.05) {$\phi_l$}; 
\node [left] at (5.7,0) {$\frak{k} :$}; 

\draw (6.1,0) -- (6.5,0); 
\draw [dotted] (6.5,0) -- (7,0); 
\draw (7,0) -- (7.4,0); 
\draw (7.6,0) -- (7.7,0.1); 
\draw (7.6,0) -- (7.7,-0.1); 
\draw (7.6,0.025) -- (8.4,0.025); 
\draw (7.6,-0.025) -- (8.4,-0.025); 
\draw (9.6,0) -- (10,0); 
\draw [dotted] (10,0) -- (10.5,0); 
\draw (10.5,0) -- (10.9,0); 
\draw (11.1,0) -- (11.2,0.1); 
\draw (11.1,0) -- (11.2,-0.1); 
\draw (11.1,0.025) -- (11.9,0.025); 
\draw (11.1,-0.025) -- (11.9,-0.025); 

\end{tikzpicture} 

Here $P_\frak{k} = \{ \phi_i + \cdots +\phi_{j-1} : 1 \le i < j \le p\} \cup \{\phi_i + \cdots +\phi_{j-1} , \phi_i + \cdots +\phi_{j-1}+2\phi_j + \cdots +2\phi_{l-1}+\phi_l : p+1 \le i < j \le l \} 
\cup \{ 2\phi_i + \cdots +2\phi_{l-1}+\phi_l : 1 \le i \le l\} \cup \{\phi_i + \cdots +\phi_{j-1}+2\phi_j + \cdots +2\phi_{l-1}+\phi_l : 1 \le i < j \le p \},$ and $P \cap \Delta_n = 
\{\phi_i + \cdots +\phi_{j-1}, \phi_i + \cdots +\phi_{j-1}+2\phi_j + \cdots +2\phi_{l-1}+\phi_l : 1 \le i  \le p < j \le l \}.$ So $\epsilon = 
2\phi_p + \cdots +2\phi_{l-1}+\phi_l, \lambda = \phi_1 + \cdots +\phi_p+2\phi_{p+1} + \cdots +2\phi_{l-1}+\phi_l.$ We have $\langle \nu , \epsilon \rangle \neq 0.$
The only possibility of $\phi' \in \Phi_0$ with $n_{\phi'} (\delta)=1$ 
is $\phi'=\phi_l.$ 

(i) Let $\phi'=\phi_l.$ Then $\phi = - w_{\frak{l}_0}^0 (\phi') = \phi_l,$ and $\nu' = -(\phi_1+ \cdots + \phi_p +\phi_{p+1}+\cdots +\phi_{l-1}).$ So $\phi_l = \epsilon + 2\nu' + 
2\phi_1 + \cdots + 2\phi_{p-1}.$ 

3. $\frak{g}=\frak{\delta}_l (l \ge 4), \frak{g}_0 = \frak{so}(2p, 2l-2p), 2 \le p \le l-2:  \Phi = \{\phi_1, \phi_2, \ldots , \phi_l\}$ with $\nu = \phi_p, \delta = 
\phi_1+2\phi_2+\cdots + 2\phi_{l-2}+\phi_{l-1}+\phi_l.$   

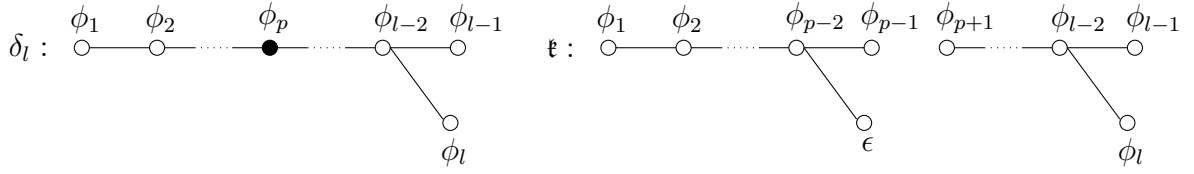
\begin{tikzpicture}

\draw  (0,0) circle [radius = 0.1]; 
\draw (1,0) circle [radius = 0.1]; 
\filldraw[black] (2.5,0) circle [radius = 0.1]; 
\draw (4,0) circle [radius = 0.1]; 
\draw (5,0) circle [radius = 0.1]; 
\draw (4.9,-1) circle [radius = 0.1]; 
\node [above] at (0.05,0.05) {$\phi_1$}; 
\node [above] at (1.05,0.05) {$\phi_2$}; 
\node [above] at (2.55,0.05) {$\phi_p$}; 
\node [above] at (4.25,0.05) {$\phi_{l-2}$}; 
\node [above] at (5.25,0.05) {$\phi_{l-1}$}; 
\node [below] at (4.95,-1.05) {$\phi_l$}; 
\node [left] at (-0.3,0) {$\frak{\delta}_l :$}; 

\draw (0.1,0) -- (0.9,0); 
\draw (1.1,0) -- (1.5,0);
\draw [dotted] (1.5,0) -- (2,0);
\draw (2,0) -- (2.4,0); 
\draw (2.6,0) -- (3,0); 
\draw [dotted] (3,0) -- (3.5,0); 
\draw (3.5,0) -- (3.9,0); 
\draw (4.1,0) -- (4.9,0); 
\draw (4.1,0) -- (4.8,-0.95); 

\draw (7,0) circle [radius = 0.1]; 
\draw (8,0) circle [radius = 0.1]; 
\draw (9.5,0) circle [radius = 0.1]; 
\draw (10.5,0) circle [radius = 0.1]; 
\draw (10.4,-1) circle [radius = 0.1]; 
\draw (11.5,0) circle [radius = 0.1]; 
\draw (13,0) circle [radius = 0.1]; 
\draw (14,0) circle [radius = 0.1]; 
\draw (13.9,-1) circle [radius = 0.1]; 

\node [above] at (7.05,0.05) {$\phi_1$}; 
\node [above] at (8.05,0.05) {$\phi_2$}; 
\node [above] at (9.75,0.05) {$\phi_{p-2}$}; 
\node [above] at (10.75,0.05) {$\phi_{p-1}$}; 
\node [below] at (10.45,-1.05) {$\epsilon$}; 
\node [above] at (11.75,0.05) {$\phi_{p+1}$}; 
\node [above] at (13.25,0.05) {$\phi_{l-2}$}; 
\node [above] at (14.25,0.05) {$\phi_{l-1}$}; 
\node [below] at (13.95,-1.05) {$\phi_l$}; 
\node [left] at (6.7,0) {$\frak{k} :$}; 

\draw (7.1,0) -- (7.9,0); 
\draw (8.1,0) -- (8.5,0); 
\draw [dotted] (8.5,0) -- (9,0); 
\draw (9,0) -- (9.4,0); 
\draw (9.6,0) -- (10.4,0); 
\draw (9.6,0) -- (10.3,-0.95);
\draw (11.6,0) -- (12,0);
\draw [dotted] (12,0) -- (12.5,0); 
\draw (12.5,0) -- (12.9,0); 
\draw (13.1,0) -- (13.9,0); 
\draw (13.1,0) -- (13.8,-0.95); 

\end{tikzpicture} 

Here $P_\frak{k} = \{ \phi_i + \cdots +\phi_{j-1} : 1 \le i < j \le p\} \cup \{\phi_i + \cdots +\phi_{j-1}, 
(\phi_i + \cdots +\phi_{l-2})+(\phi_j + \cdots +\phi_l) : p+1 \le i < j \le l \} \cup 
\{(\phi_i + \cdots +\phi_{l-2})+(\phi_j + \cdots+\phi_l) : 1 \le i < j \le p \},$ and $P \cap \Delta_n = 
\{\phi_i + \cdots +\phi_{j-1}, (\phi_i + \cdots +\phi_{l-2})+(\phi_j + \cdots +\phi_l) : 1 \le i  \le p < j \le l \}.$ So $\epsilon = 
\phi_{p-1}+2\phi_p + \cdots +2\phi_{l-2}+\phi_{l-1}+\phi_l, \lambda = \phi_1 + \cdots +\phi_p+2\phi_{p+1} + \cdots +2\phi_{l-2}+\phi_{l-1}+\phi_l.$ 
We have $\langle \nu , \epsilon \rangle \neq 0$ and 
$\langle \lambda , \phi_1 \rangle \neq 0.$ The only possibilities of $\phi' \in \Phi_0$ with $n_{\phi'} (\delta)=1$ are $\phi'=\phi_1, \phi_{l-1},$ or $\phi_l.$ 

(i) Let $\phi'=\phi_1.$ Then $\phi = - w_{\frak{l}_0}^0 (\phi') = \phi_{p-1},$ and 
$\nu' = -(\phi_p +2\phi_{p+1}+\cdots +2\phi_{l-2}+\phi_{l-1}+\phi_l).$ 
So $\phi_{p-1} = \epsilon + 2\nu' + 2\phi_{p+1} + \cdots + 2\phi_{l-2}+\phi_{l-1}+\phi_l.$ 

If $\phi'=\phi_{l-1},$ or $\phi_l,$ then $\phi = - w_{\frak{l}_0}^0 (\phi') = \phi_{l-1} \textrm{ or } \phi_l.$

(ii) Let $\phi = \phi_l.$ Then $\nu' = -(\phi_1+\cdots+\phi_p +\phi_{p+1}+\cdots +\phi_{l-2}+\phi_{l-1}).$ 
So $\phi_l = \epsilon + 2\nu' + 2\phi_1+\cdots+2\phi_{p-2}+\phi_{p-1}+\phi_{l-1}.$ 

(iii) Let $\phi = \phi_{l-1}.$ Then $\nu' = -(\phi_1+\cdots +\phi_p +\phi_{p+1}+\cdots +\phi_{l-2}+\phi_l).$ 
So $\phi_{l-1} = \epsilon + 2\nu' + 2\phi_1+\cdots+2\phi_{p-2}+\phi_{p-1}+\phi_l.$ 

4. $\frak{g}=\frak{e}_6, \frak{g}_0 = \frak{e}_{6(2)}:  \Phi = \{\phi_1,\phi_2,\phi_3,\phi_4,\phi_5,\phi_6\}$ with $\nu = \phi_2, \delta= 
\phi_1+2\phi_2+2\phi_3+3\phi_4+2\phi_5+\phi_6.$  

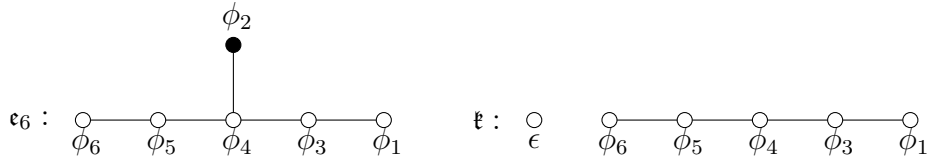
\begin{tikzpicture}

\draw  (0,0) circle [radius = 0.1]; 
\draw (1,0) circle [radius = 0.1]; 
\draw (2,0) circle [radius = 0.1]; 
\filldraw[black] (2,1) circle [radius = 0.1]; 
\draw (3,0) circle [radius = 0.1]; 
\draw (4,0) circle [radius = 0.1]; 
\node [below] at (0.05,0.05) {$\phi_6$}; 
\node [below] at (1.05,0.05) {$\phi_5$}; 
\node [below] at (2.05,0.05) {$\phi_4$}; 
\node [above] at (2.05,1.05) {$\phi_2$}; 
\node [below] at (3.05,0.05) {$\phi_3$}; 
\node [below] at (4.05,0.05) {$\phi_1$}; 
\node [left] at (-0.3,0) {$\frak{e}_6 :$}; 

\draw (0.1,0) -- (0.9,0); 
\draw (1.1,0) -- (1.9,0);
\draw (2.1,0) -- (2.9,0); 
\draw (2,0.1) -- (2,0.9); 
\draw (3.1,0) -- (3.9,0); 

\draw (6,0) circle [radius = 0.1]; 
\draw (7,0) circle [radius = 0.1]; 
\draw (8,0) circle [radius = 0.1]; 
\draw (9,0) circle [radius = 0.1]; 
\draw (10,0) circle [radius = 0.1]; 
\draw (11,0) circle [radius = 0.1]; 

\node [below] at (6,-0.05) {$\epsilon$}; 
\node [below] at (7.05,0.05) {$\phi_6$}; 
\node [below] at (8.05,0.05) {$\phi_5$}; 
\node [below] at (9.05,0.05) {$\phi_4$}; 
\node [below] at (10.05,0.05) {$\phi_3$}; 
\node [below] at (11.05,0.05) {$\phi_1$}; 
\node [left] at (5.7,0) {$\frak{k} :$}; 

\draw (7.1,0) -- (7.9,0); 
\draw (8.1,0) -- (8.9,0); 
\draw (9.1,0) -- (9.9,0); 
\draw (10.1,0) -- (10.9,0); 

\end{tikzpicture} 

Here $P_\frak{k} = \{\phi_1, \phi_1 + \phi_3 +\cdots \phi_{j-1}, \phi_i + \cdots +\phi_{j-1} : 3 \le i < j \le 7\} \cup \{\delta\},$ and 
$P \cap \Delta_n = \{ \phi_2, \phi_2+\phi_4+\cdots +\phi_i,  \phi_2+\phi_3+\phi_4+\cdots +\phi_i, \phi_1+\phi_2+\phi_3+\phi_4+\cdots +\phi_i : 4 \le i \le 6\} \cup 
\{\phi_2+\phi_3+2\phi_4+\phi_5+\cdots +\phi_i, \phi_1+\phi_2+\phi_3+2\phi_4+\phi_5+\cdots +\phi_i, \phi_1+\phi_2+2\phi_3+2\phi_4+\phi_5+\cdots +\phi_i : 5 \le i \le 6\} 
\cup \{\phi_2+\phi_3+2\phi_4+2\phi_5+\phi_6, \phi_1+\phi_2+\phi_3+2\phi_4+2\phi_5+\phi_6, \phi_1+\phi_2+2\phi_3+2\phi_4+2\phi_5+\phi_6, 
\phi_1+\phi_2+2\phi_3+3\phi_4+2\phi_5+\phi_6\}.$ So $\epsilon = \delta, \lambda = \phi_1 + \phi_2 +2\phi_3++3\phi_4+2\phi_5 +\phi_6.$ 
We have $\langle \nu , \epsilon \rangle \neq 0.$
The only possibilities of $\phi' \in \Phi_0$ with $n_{\phi'} (\delta)=1$ are $\phi'=\phi_1,$ or $\phi_6.$ 

(i) Let $\phi'=\phi_1.$ Then $\phi = - w_{\frak{l}_0}^0 (\phi') = \phi_6,$ and $\nu' = -(\phi_1+\phi_2+2\phi_3+2\phi_4+\phi_5).$ 
So $\phi_6 = \epsilon + 2\nu' +\phi_1+2\phi_3+\phi_4.$ 

(ii) Let $\phi' = \phi_6.$ Then $\phi = - w_{\frak{l}_0}^0 (\phi') = \phi_1,$ and $\nu' = -(\phi_2+\phi_3+2\phi_4+2\phi_5+\phi_6).$ 
So $\phi_1 = \epsilon + 2\nu' + \phi_4+2\phi_5+\phi_6.$ 

5. $\frak{g}=\frak{e}_7, \frak{g}_0 = \frak{e}_{7(-5)}:  \Phi = \{\phi_1,\phi_2,\phi_3,\phi_4,\phi_5,\phi_6,\phi_7\}$ with $\nu = \phi_1, \delta= 
2\phi_1+2\phi_2+3\phi_3+4\phi_4+3\phi_5+2\phi_6+\phi_7.$   

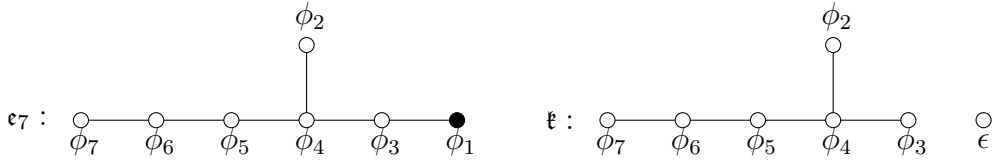
\begin{tikzpicture}

\draw  (0,0) circle [radius = 0.1]; 
\draw (1,0) circle [radius = 0.1]; 
\draw (2,0) circle [radius = 0.1]; 
\draw (3,0) circle [radius = 0.1]; 
\draw (3,1) circle [radius = 0.1]; 
\draw (4,0) circle [radius = 0.1]; 
\filldraw[black] (5,0) circle [radius = 0.1]; 
\node [below] at (0.05,0.05) {$\phi_7$}; 
\node [below] at (1.05,0.05) {$\phi_6$}; 
\node [below] at (2.05,0.05) {$\phi_5$}; 
\node [below] at (3.05,0.05) {$\phi_4$}; 
\node [above] at (3.05,1.05) {$\phi_2$}; 
\node [below] at (4.05,0.05) {$\phi_3$}; 
\node [below] at (5.05,0.05) {$\phi_1$}; 
\node [left] at (-0.3,0) {$\frak{e}_7 :$}; 

\draw (0.1,0) -- (0.9,0); 
\draw (1.1,0) -- (1.9,0);
\draw (2.1,0) -- (2.9,0); 
\draw (3,0.1) -- (3,0.9); 
\draw (3.1,0) -- (3.9,0); 
\draw (4.1,0) -- (4.9,0); 

\draw (7,0) circle [radius = 0.1]; 
\draw (8,0) circle [radius = 0.1]; 
\draw (9,0) circle [radius = 0.1]; 
\draw (10,0) circle [radius = 0.1]; 
\draw (10,1) circle [radius = 0.1]; 
\draw (11,0) circle [radius = 0.1]; 
\draw (12,0) circle [radius = 0.1]; 

\node [below] at (7.05,0.05) {$\phi_7$}; 
\node [below] at (8.05,0.05) {$\phi_6$}; 
\node [below] at (9.05,0.05) {$\phi_5$}; 
\node [below] at (10.05,0.05) {$\phi_4$}; 
\node [above] at (10.05,1.05) {$\phi_2$}; 
\node [below] at (11.05,0.05) {$\phi_3$}; 
\node [below] at (12,-0.05) {$\epsilon$}; 
\node [left] at (6.7,0) {$\frak{k} :$}; 

\draw (7.1,0) -- (7.9,0); 
\draw (8.1,0) -- (8.9,0); 
\draw (9.1,0) -- (9.9,0); 
\draw (10.1,0) -- (10.9,0); 
\draw (10,0.1) -- (10,0.9); 

\end{tikzpicture} 

Here $P_\frak{k} = \{ \phi_{i+1} + \cdots +\phi_j, (\phi_2+\cdots+\phi_i)+ (\phi_4+\phi_5+\cdots +\phi_j) : 2 \le i < j \le 7\} \cup \{\delta\},$ and 
$P \cap \Delta_n = \{ \phi_1, \phi_1+\phi_3+\cdots +\phi_i : 3 \le i \le 7\} \cup \{\phi_1+\phi_2+\phi_3+\phi_4+\cdots+\phi_i : 4 \le i \le 7\}\cup 
\{\phi_1+\phi_2+\phi_3+2\phi_4+\phi_5+\cdots +\phi_i , \phi_1+ \phi_2+2\phi_3+2\phi_4+\phi_5+ \cdots + \phi_i : 5 \le i \le 7\} \cup 
\{\phi_1+\phi_2+\phi_3+2\phi_4+2\phi_5+\phi_6+\cdots +\phi_i , \phi_1+\phi_2+2\phi_3+2\phi_4+2\phi_5+\phi_6+\cdots +\phi_i , 
\phi_1+\phi_2+2\phi_3+3\phi_4+2\phi_5+\phi_6+\cdots +\phi_i , \phi_1+2\phi_2+2\phi_3+3\phi_4+2\phi_5+\phi_6+\cdots +\phi_i : 6 \le i \le 7\} \cup 
\{\phi_1+\phi_2+\phi_3+2\phi_4+2\phi_5+2\phi_6+\phi_7, \phi_1+\phi_2+2\phi_3+2\phi_4+2\phi_5+2\phi_6+\phi_7,
\phi_1+\phi_2+2\phi_3+3\phi_4+2\phi_5+2\phi_6+\phi_7, \phi_1+\phi_2+2\phi_3+3\phi_4+3\phi_5+2\phi_6+\phi_7, 
\phi_1+2\phi_2+2\phi_3+3\phi_4+2\phi_5+2\phi_6+\phi_7, \phi_1+2\phi_2+2\phi_3+3\phi_4+3\phi_5+2\phi_6+\phi_7, 
\phi_1+2\phi_2+2\phi_3+4\phi_4+3\phi_5+2\phi_6+\phi_7, \phi_1+2\phi_2+3\phi_3+4\phi_4+3\phi_5+2\phi_6+\phi_7\}.$ So $\epsilon = \delta, 
\lambda = \phi_1+2\phi_2+3\phi_3+4\phi_4+3\phi_5+2\phi_6+\phi_7.$ We have $\langle \nu , \epsilon \rangle \neq 0.$ 
The only possibility of $\phi' \in \Phi_0$ with $n_{\phi'} (\delta)=1$ is $\phi'=\phi_7.$ 

(i) Let $\phi'=\phi_7.$ Then $\phi = - w_{\frak{l}_0}^0 (\phi') = \phi_7,$ and $\nu' = -(\phi_1+2\phi_2+2\phi_3+3\phi_4+2\phi_5+\phi_6).$ 
So $\phi_7 = \epsilon + 2\nu' +2\phi_2+\phi_3+2\phi_4+\phi_5.$ 

6. $\frak{g}=\frak{e}_7, \frak{g}_0 = \frak{e}_{7(7)}:  \Phi = \{\phi_1,\phi_2,\phi_3,\phi_4,\phi_5,\phi_6,\phi_7\}$ with $\nu = \phi_2, \delta = 
2\phi_1+2\phi_2+3\phi_3+4\phi_4+3\phi_5+2\phi_6+\phi_7.$   

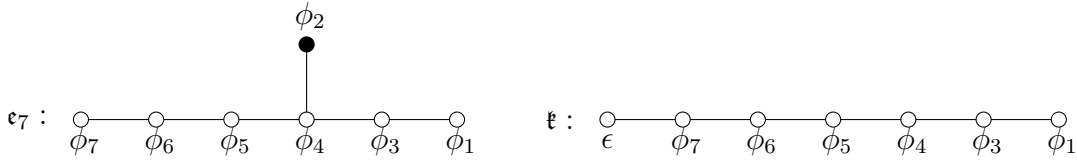
\begin{tikzpicture}

\draw  (0,0) circle [radius = 0.1]; 
\draw (1,0) circle [radius = 0.1]; 
\draw (2,0) circle [radius = 0.1]; 
\draw (3,0) circle [radius = 0.1]; 
\filldraw [black] (3,1) circle [radius = 0.1]; 
\draw (4,0) circle [radius = 0.1]; 
\draw (5,0) circle [radius = 0.1]; 
\node [below] at (0.05,0.05) {$\phi_7$}; 
\node [below] at (1.05,0.05) {$\phi_6$}; 
\node [below] at (2.05,0.05) {$\phi_5$}; 
\node [below] at (3.05,0.05) {$\phi_4$}; 
\node [above] at (3.05,1.05) {$\phi_2$}; 
\node [below] at (4.05,0.05) {$\phi_3$}; 
\node [below] at (5.05,0.05) {$\phi_1$}; 
\node [left] at (-0.3,0) {$\frak{e}_7 :$}; 

\draw (0.1,0) -- (0.9,0); 
\draw (1.1,0) -- (1.9,0);
\draw (2.1,0) -- (2.9,0); 
\draw (3,0.1) -- (3,0.9); 
\draw (3.1,0) -- (3.9,0); 
\draw (4.1,0) -- (4.9,0); 

\draw (7,0) circle [radius = 0.1]; 
\draw (8,0) circle [radius = 0.1]; 
\draw (9,0) circle [radius = 0.1]; 
\draw (10,0) circle [radius = 0.1]; 
\draw (11,0) circle [radius = 0.1]; 
\draw (12,0) circle [radius = 0.1]; 
\draw (13,0) circle [radius = 0.1]; 

\node [below] at (7,-0.05) {$\epsilon$}; 
\node [below] at (8.05,0.05) {$\phi_7$}; 
\node [below] at (9.05,0.05) {$\phi_6$}; 
\node [below] at (10.05,0.05) {$\phi_5$}; 
\node [below] at (11.05,0.05) {$\phi_4$}; 
\node [below] at (12.05,0.05) {$\phi_3$}; 
\node [below] at (13.05,0.05) {$\phi_1$}; 
\node [left] at (6.7,0) {$\frak{k} :$}; 

\draw (7.1,0) -- (7.9,0); 
\draw (8.1,0) -- (8.9,0); 
\draw (9.1,0) -- (9.9,0); 
\draw (10.1,0) -- (10.9,0); 
\draw (11.1,0) -- (11.9,0); 
\draw (12.1,0) -- (12.9,0); 
\end{tikzpicture} 

Here $P_\frak{k} = \{ \phi_1, \phi_1+\phi_3+\cdots + \phi_{j-1}, \phi_i + \cdots + \phi_{j-1} : 3 \le i < j \le 8\} \cup 
\{\phi_i + \cdots + \phi_7 + (\phi_1+2\phi_2+2\phi_3+3\phi_4+2\phi_5+\phi_6), \phi_1+\phi_3+\cdots + \phi_7 + (\phi_1+2\phi_2+2\phi_3+3\phi_4+2\phi_5+\phi_6) : 3 \le i \le 7\},
$ and $P \cap \Delta_n = \{\phi_2, \phi_2+\phi_4+\cdots + \phi_i, \phi_2+\phi_3+\phi_4+\cdots +\phi_i, \phi_1+\phi_2+\phi_3+\phi_4+\cdots +\phi_i : 4 \le i \le 7\} \cup 
\{\phi_2+\phi_3+2\phi_4+\phi_5+\cdots +\phi_i, \phi_1+\phi_2+\phi_3+2\phi_4+\phi_5+\cdots +\phi_i, \phi_1+\phi_2+2\phi_3+2\phi_4+\phi_5+\cdots +\phi_i : 5 \le i \le 7\} \cup 
\{\phi_2+\phi_3+2\phi_4+2\phi_5+\phi_6+\cdots +\phi_i, \phi_1+\phi_2+\phi_3+2\phi_4+2\phi_5+\phi_6+\cdots +\phi_i, \phi_1+\phi_2+2\phi_3+2\phi_4+2\phi_5+\phi_6+
\cdots +\phi_i, \phi_1+\phi_2+2\phi_3+3\phi_4+2\phi_5+\phi_6+\cdots +\phi_i : 6\le i \le 7 \} \cup \{\phi_2+\phi_3+2\phi_4+2\phi_5+2\phi_6+\phi_7, 
\phi_1+\phi_2+\phi_3+2\phi_4+2\phi_5+2\phi_6+\phi_7, \phi_1+\phi_2+2\phi_3+2\phi_4+2\phi_5+2\phi_6+\phi_7, \phi_1+\phi_2+2\phi_3+3\phi_4+2\phi_5+2\phi_6+\phi_7, 
\phi_1+\phi_2+2\phi_3+3\phi_4+3\phi_5+2\phi_6+\phi_7\}.$ 
So $\epsilon = \phi_1+2\phi_2+2\phi_3+3\phi_4+2\phi_5
+\phi_6, \lambda = \phi_1+\phi_2+2\phi_3+3\phi_4+3\phi_5+2\phi_6+\phi_7.$ We have $\langle \nu , \epsilon \rangle \neq 0$ and 
$\langle \lambda , \phi_5 \rangle \neq 0.$ The only possibility of $\phi' \in \Phi_0$ with $n_{\phi'} (\delta)=1$ is $\phi'=\phi_7.$ 

(i) Let $\phi'=\phi_7.$ Then $\phi = - w_{\frak{l}_0}^0 (\phi') = \phi_1,$ and $\nu' = -(\phi_2+\phi_3+2\phi_4+2\phi_5+2\phi_6+\phi_7).$ 
So $\phi_1 = \epsilon + 2\nu' +\phi_4+2\phi_5+3\phi_6+2\phi_7.$ 

If $\frak{g}=\frak{e}_8, \frak{f}_4,$ or $\frak{g}_2;$ there are no $\phi' \in \Phi$ with $n_{\phi'}(\delta)=1.$ 

\noindent
{\bf Proof of Corollary \ref{cor}:}
Let $A$ be the number of Borel-de Siebenthal positive root systems of $\Delta$ containing $P_\frak{k}.$ 
Then from Theorem \ref{main}, we have $A = 1 + B,$ where $B$ is the number of simple roots in $P$ whose coefficient in $\delta,$ 
when expressed as the sum of simple roots, is $1.$ Thus 
\[ A= 
\begin{cases}
1  & \textrm{if } \frak{g} = \frak{e}_8, \frak{f}_4, \frak{g}_2, \\
2  & \textrm{if } \frak{g} = \frak{b}_l(l \ge 2), \frak{c}_l (l \ge 2), \frak{e}_7, \\
3  & \textrm{if } \frak{g} = \frak{e}_6, \\
4  & \textrm{if } \frak{g} = \frak{\delta}_l (l \ge 4); \\
\end{cases}
\] 
which is equal to the covering index of $Int(\frak{u}),$ where $\frak{u}$ is the compact real form of $\frak{g}$ \cite[Th. 3.32, Ch. X]{helgason}. Since the covering index of 
$Int(\frak{g})=$ the covering index of $Int(\frak{u}),$ the proof is complete. 

\begin{remark}\label{rk}
(i) Let $\frak{k}$ be semisimple, $\phi' \in \Phi_0$ be such that $n_{\phi'}(\delta)=1,$ 
$\phi \in \Phi_0$ be such that $w_{\frak{l}_0}^0 (\phi') = -\phi,$ and $P'$ be the Borel-de Siebenthal positive root system of 
$\Delta$ whose simple roots are given by $(\Phi_\frak{k} \setminus \{\phi \}) \cup \{\nu'\},$ where $\nu'$ be the lowest weight of the irreducible $\frak{l}'_0$-submodule 
of $\frak{p}$ with highest weight $\lambda,$ and $\frak{l}'_0$ be the reductive subalgebra of $\frak{k}$ containing $\frak{h}$ and the Dynkin diagram of $[\frak{l}'_0, \frak{l}'_0]$ be 
the subdiagram of the Dynkin diagram of $\frak{k}$ with vertices $\Phi_\frak{k} \setminus \{\phi\}.$ 
Then $P' \cap \Delta_n = \{\beta \in \Delta_n : n_\phi(\beta) = 1 \} \cup \{\beta \in (-P) \cap \Delta_n : n_\phi(\beta) = 0 \}.$ 

(ii) If $\frak{k}$ has non-zero centre, then obviously there are exactly two Borel-de Siebenthal positive root systems of $\Delta$ containing $P_\frak{k} = P_0.$ The corresponding 
sets of simple roots are given by $\Phi_0 \cup \{\nu\},$ and $\Phi_0 \cup \{-\delta\}.$ 
\end{remark}

The $\frak{l}'_0$-submodule of $\frak{p}$ with highest weight $\lambda$ and Remark \ref{rk}(i) are illustrated with the following example. 

{\bf Example :} Let $\frak{g}_0 = \frak{so}(2p, 2l-2p+1), 2 \le p \le l.$ Assume that $\Phi = \{\phi_1, \phi_2, \ldots , \phi_l\}$ with $\nu = \phi_p,$ and $\delta = 
\phi_1+2\phi_2 + \cdots +2\phi_l.$ Recall that $\epsilon = \phi_{p-1}+2\phi_p + \cdots +2\phi_l,$ and $\lambda = \phi_1 + \cdots +\phi_p+2\phi_{p+1} + \cdots +2\phi_l.$ 
The only possibility of $\phi' \in \Phi_0$ with $n_{\phi'} (\delta)=1$ is $\phi'=\phi_1.$ Then $\phi = - w_{\frak{l}_0}^0 (\phi') = \phi_{p-1},$ and 
$\nu' = -(\phi_p+2\phi_{p+1} + \cdots +2\phi_l).$ The simple roots of the reductive 
subalgebra $\frak{l}'_0$ of $\frak{k}$ are $\phi_1,\ldots , \phi_{p-2}, \epsilon, \phi_{p+1},\ldots , \phi_l.$ The weights of the irreducible $\frak{k}$-module $\frak{p}$ are the 
vertices of the diagram in the Figure \ref{diagram}. 

\begin{figure}[!h]
\begin{center} 
\begin{tikzpicture}

\filldraw [black] (0,0) circle [radius = 0.1]; 

\filldraw [black] (-1,-1) circle [radius = 0.1]; 
\filldraw [black] (1,-1) circle [radius = 0.1]; 
\filldraw [black] (-2.5,-2.5) circle [radius = 0.1]; 
\filldraw [black] (3,-3) circle [radius = 0.1]; 
\filldraw [black] (-3.5,-3.5) circle [radius = 0.1]; 
\filldraw [black] (4,-4) circle [radius = 0.1]; 
\filldraw [black] (5,-5) circle [radius = 0.1]; 
\filldraw [black] (7,-7) circle [radius = 0.1];
\filldraw [black] (8,-8) circle [radius = 0.1];

\filldraw [black] (0,-2) circle [radius = 0.1]; 
\filldraw [black] (-1.5,-3.5) circle [radius = 0.1]; 
\filldraw [black] (-2.5,-4.5) circle [radius = 0.1]; 

\filldraw [black] (2,-4) circle [radius = 0.1]; 
\filldraw [black] (3,-5) circle [radius = 0.1]; 
\filldraw [black] (4,-6) circle [radius = 0.1]; 

\filldraw [black] (0.5,-5.5) circle [radius = 0.1]; 
\filldraw [black] (-0.5,-6.5) circle [radius = 0.1]; 
\filldraw [black] (1.5,-6.5) circle [radius = 0.1]; 
\filldraw [black] (0.5,-7.5) circle [radius = 0.1]; 
\filldraw [black] (2.5,-7.5) circle [radius = 0.1]; 
\filldraw [black] (1.5,-8.5) circle [radius = 0.1]; 

\filldraw [black] (6,-8) circle [radius = 0.1];
\filldraw [black] (7,-9) circle [radius = 0.1];

\filldraw [black] (4.5,-9.5) circle [radius = 0.1];
\filldraw [black] (3.5,-10.5) circle [radius = 0.1];
\filldraw [black] (5.5,-10.5) circle [radius = 0.1];
\filldraw [black] (4.5,-11.5) circle [radius = 0.1];

\filldraw [black] (-5,-3.5) circle [radius = 0.1]; 

\filldraw [black] (-6,-4.5) circle [radius = 0.1]; 
\filldraw [black] (-4,-4.5) circle [radius = 0.1]; 
\filldraw [black] (-7.5,-6) circle [radius = 0.1]; 
\filldraw [black] (-2,-6.5) circle [radius = 0.1]; 
\filldraw [black] (-8.5,-7) circle [radius = 0.1]; 
\filldraw [black] (-1,-7.5) circle [radius = 0.1]; 
\filldraw [black] (0,-8.5) circle [radius = 0.1]; 
\filldraw [black] (2,-10.5) circle [radius = 0.1];
\filldraw [black] (3,-11.5) circle [radius = 0.1];

\filldraw [black] (-5,-5.5) circle [radius = 0.1]; 
\filldraw [black] (-6.5,-7) circle [radius = 0.1]; 
\filldraw [black] (-7.5,-8) circle [radius = 0.1]; 

\filldraw [black] (-3,-7.5) circle [radius = 0.1]; 
\filldraw [black] (-2,-8.5) circle [radius = 0.1]; 
\filldraw [black] (-1,-9.5) circle [radius = 0.1]; 

\filldraw [black] (-4.5,-9) circle [radius = 0.1]; 
\filldraw [black] (-5.5,-10) circle [radius = 0.1]; 
\filldraw [black] (-3.5,-10) circle [radius = 0.1]; 
\filldraw [black] (-4.5,-11) circle [radius = 0.1]; 
\filldraw [black] (-2.5,-11) circle [radius = 0.1]; 
\filldraw [black] (-3.5,-12) circle [radius = 0.1]; 

\filldraw [black] (1,-11.5) circle [radius = 0.1];
\filldraw [black] (2,-12.5) circle [radius = 0.1];

\filldraw [black] (-0.5,-13) circle [radius = 0.1];
\filldraw [black] (-1.5,-14) circle [radius = 0.1];
\filldraw [black] (0.5,-14) circle [radius = 0.1];
\filldraw [black] (-0.5,-15) circle [radius = 0.1];

\node [above] at (0,0) {$\lambda$}; 
\node [below] at (4.5,-11.5) {$\nu$}; 
\node [above] at (-5,-3.5) {$-\nu$}; 
\node [below] at (-0.5,-15) {$-\lambda$}; 
\node [below] at (3,-11.5) {$\nu'$}; 

\node [left] at (-0.3,-0.3) {$\phi_1$};
\node [right] at (0.4,-1.6) {$\phi_1$};
\node [left] at (2.7,-3.3) {$\phi_1$};
\node [right] at (3.4,-4.6) {$\phi_1$};
\node [right] at (4.4,-5.6) {$\phi_1$};
\node [left] at (6.7,-7.3) {$\phi_1$};
\node [right] at (7.4,-8.6) {$\phi_1$};

\node [left] at (-2.8,-2.8) {$\phi_{p-1}$};
\node [right] at (-2.1,-4.1) {$\phi_{p-1}$};
\node [left] at (0.2,-5.8) {$\phi_{p-1}$};
\node [right] at (0.9,-7.1) {$\phi_{p-1}$};
\node [right] at (1.9,-8.1) {$\phi_{p-1}$};
\node [left] at (4.2,-9.8) {$\phi_{p-1}$};
\node [right] at (4.9,-11.1) {$\phi_{p-1}$}; 

\node [right] at (0.3,-0.3) {$\phi_{p+1}$};
\node [right] at (-0.7,-1.3) {$\phi_{p+1}$};
\node [right] at (-2.2,-2.8) {$\phi_{p+1}$};
\node [right] at (-3.2,-3.8) {$\phi_{p+1}$};

\node [right] at (3.3,-3.3) {$\phi_l$};
\node [right] at (2.3,-4.3) {$\phi_l$};
\node [right] at (0.8,-5.8) {$\phi_l$};
\node [right] at (-0.2,-6.8) {$\phi_l$};

\node [right] at (4.3,-4.3) {$\phi_l$};
\node [right] at (3.3,-5.3) {$\phi_l$};
\node [right] at (1.8,-6.8) {$\phi_l$};
\node [right] at (0.8,-7.8) {$\phi_l$};

\node [right] at (7.3,-7.3) {$\phi_{p+1}$};
\node [right] at (6.3,-8.3) {$\phi_{p+1}$};
\node [right] at (4.8,-9.8) {$\phi_{p+1}$};
\node [right] at (3.8,-10.8) {$\phi_{p+1}$};

\node [above] at (-3,-2.7) {$\epsilon$};
\node [above] at (-2,-3.7) {$\epsilon$};
\node [above] at (0,-5.7) {$\epsilon$};
\node [above] at (1,-6.7) {$\epsilon$};
\node [above] at (2,-7.7) {$\epsilon$};
\node [above] at (4,-9.7) {$\epsilon$};
\node [above] at (5,-10.7) {$\epsilon$};

\node [above] at (-4,-3.7) {$\epsilon$};
\node [above] at (-3,-4.7) {$\epsilon$};
\node [above] at (-1,-6.7) {$\epsilon$};
\node [above] at (0,-7.7) {$\epsilon$};
\node [above] at (1,-8.7) {$\epsilon$};
\node [above] at (3,-10.7) {$\epsilon$};
\node [above] at (4,-11.7) {$\epsilon$};

\node [left] at (-5.3,-3.8) {$\phi_{p-1}$};
\node [right] at (-4.6,-5.1) {$\phi_{p-1}$};
\node [left] at (-2.3,-6.8) {$\phi_{p-1}$};
\node [right] at (-1.6,-8.1) {$\phi_{p-1}$};
\node [right] at (-0.6,-9.1) {$\phi_{p-1}$};
\node [left] at (1.7,-10.8) {$\phi_{p-1}$};
\node [right] at (2.4,-12.1) {$\phi_{p-1}$};

\node [left] at (-7.8,-6.3) {$\phi_1$};
\node [right] at (-7.1,-7.6) {$\phi_1$};
\node [left] at (-4.8,-9.3) {$\phi_1$};
\node [right] at (-4.1,-10.6) {$\phi_1$};
\node [right] at (-3.1,-11.6) {$\phi_1$};
\node [left] at (-0.8,-13.3) {$\phi_1$};
\node [right] at (-0.1,-14.6) {$\phi_1$}; 

\node [right] at (-4.7,-3.8) {$\phi_{p+1}$};
\node [right] at (-5.7,-4.8) {$\phi_{p+1}$};
\node [right] at (-7.2,-6.3) {$\phi_{p+1}$};
\node [right] at (-8.2,-7.3) {$\phi_{p+1}$};

\node [right] at (-1.7,-6.8) {$\phi_l$};
\node [right] at (-2.7,-7.8) {$\phi_l$};
\node [right] at (-4.2,-9.3) {$\phi_l$};
\node [right] at (-5.2,-10.3) {$\phi_l$};

\node [right] at (-0.7,-7.8) {$\phi_l$};
\node [right] at (-1.7,-8.8) {$\phi_l$};
\node [right] at (-3.2,-10.3) {$\phi_l$};
\node [right] at (-4.2,-11.3) {$\phi_l$};

\node [right] at (2.3,-10.8) {$\phi_{p+1}$};
\node [right] at (1.3,-11.8) {$\phi_{p+1}$};
\node [right] at (-0.2,-13.3) {$\phi_{p+1}$};
\node [right] at (-1.2,-14.3) {$\phi_{p+1}$};

\draw (0,-0.1) -- (-0.95,-0.9); 
\draw (0,-0.1) -- (0.95,-0.9);
\draw [dotted] (-1.1,-1.1) -- (-2.4,-2.4);
\draw [dotted](1.1,-1.1) -- (2.9,-2.9);
\draw (-2.55,-2.6) -- (-3.45,-3.4);
\draw (3.05,-3.1) -- (3.95,-3.9);
\draw (4.05,-4.1) -- (4.95,-4.9);
\draw [dotted] (5.05,-5.1) -- (6.95,-6.9);
\draw (7.05,-7.1) -- (7.95,-7.9);

\draw (-1,-1.1) -- (-0.05,-1.9);
\draw [dotted](0.1,-2.1) -- (1.9,-3.9);
\draw (2.05,-4.1) -- (2.95,-4.9);
\draw (3.05,-5.1) -- (3.95,-5.9);
\draw [dotted] (4.05,-6.1) -- (5.95,-7.9);
\draw (6.05,-8.1) -- (6.95,-8.9);

\draw (-2.5,-2.6) -- (-1.55,-3.4);
\draw [dotted](-1.4,-3.6) -- (0.4,-5.4);
\draw (0.55,-5.6) -- (1.45,-6.4);
\draw (1.55,-6.6) -- (2.45,-7.4);
\draw [dotted] (2.55,-7.6) -- (4.45,-9.4);
\draw (4.55,-9.6) -- (5.45,-10.4);

\draw (-3.5,-3.6) -- (-2.55,-4.4);
\draw [dotted](-2.4,-4.6) -- (-0.6,-6.4);
\draw (-0.45,-6.6) -- (0.45,-7.4);
\draw (0.55,-7.6) -- (1.45,-8.4);
\draw [dotted] (1.55,-8.6) -- (3.45,-10.4);
\draw (3.55,-10.6) -- (4.45,-11.4);

\draw (1,-1.1) -- (0.05,-1.9); 
\draw [dotted] (-0.1,-2.1) -- (-1.4,-3.4);
\draw (-1.55,-3.6) -- (-2.45,-4.4);

\draw (3,-3.1) -- (2.05,-3.9); 
\draw [dotted] (1.9,-4.1) -- (0.6,-5.4);
\draw (0.45,-5.6) -- (-0.45,-6.4);

\draw (4,-4.1) -- (3.05,-4.9); 
\draw [dotted] (2.9,-5.1) -- (1.6,-6.4);
\draw (1.45,-6.6) -- (0.55,-7.4);

\draw (5,-5.1) -- (4.05,-5.9); 
\draw [dotted] (3.9,-6.1) -- (2.6,-7.4);
\draw (2.45,-7.6) -- (1.55,-8.4);

\draw (7,-7.1) -- (6.05,-7.9); 
\draw [dotted] (5.9,-8.1) -- (4.6,-9.4);
\draw (4.45,-9.6) -- (3.55,-10.4);

\draw (8,-8.1) -- (7.05,-8.9); 
\draw [dotted] (6.9,-9.1) -- (5.6,-10.4);
\draw (5.45,-10.6) -- (4.55,-11.4);

\draw (-5,-3.6) -- (-5.95,-4.4); 
\draw (-5,-3.6) -- (-4.05,-4.4);
\draw [dotted] (-6.1,-4.6) -- (-7.4,-5.9);
\draw [dotted](-3.9,-4.6) -- (-2.1,-6.4);
\draw (-7.55,-6.1) -- (-8.45,-6.9);
\draw (-1.95,-6.6) -- (-1.05,-7.4);
\draw (-0.95,-7.6) -- (-0.05,-8.4);
\draw [dotted] (0.05,-8.6) -- (1.95,-10.4);
\draw (2.05,-10.6) -- (2.95,-11.4);

\draw (-6,-4.6) -- (-5.05,-5.4);
\draw [dotted](-4.9,-5.6) -- (-3.1,-7.4);
\draw (-2.95,-7.6) -- (-2.05,-8.4);
\draw (-1.95,-8.6) -- (-1.05,-9.4);
\draw [dotted] (-0.95,-9.6) -- (0.95,-11.4);
\draw (1.05,-11.6) -- (1.95,-12.4);

\draw (-7.5,-6.1) -- (-6.55,-6.9);
\draw [dotted](-6.4,-7.1) -- (-4.6,-8.9);
\draw (-4.45,-9.1) -- (-3.55,-9.9);
\draw (-3.45,-10.1) -- (-2.55,-10.9);
\draw [dotted] (-2.45,-11.1) -- (-0.55,-12.9);
\draw (-0.45,-13.1) -- (0.45,-13.9);

\draw (-8.5,-7.1) -- (-7.55,-7.9);
\draw [dotted](-7.4,-8.1) -- (-5.6,-9.9);
\draw (-5.45,-10.1) -- (-4.55,-10.9);
\draw (-4.45,-11.1) -- (-3.55,-11.9);
\draw [dotted] (-3.45,-12.1) -- (-1.55,-13.9);
\draw (-1.45,-14.1) -- (-0.55,-14.9);

\draw (-4,-4.6) -- (-4.95,-5.4); 
\draw [dotted] (-5.1,-5.6) -- (-6.4,-6.9);
\draw (-6.55,-7.1) -- (-7.45,-7.9);

\draw (-2,-6.6) -- (-2.95,-7.4); 
\draw [dotted] (-3.1,-7.6) -- (-4.4,-8.9);
\draw (-4.55,-9.1) -- (-5.45,-9.9);

\draw (-1,-7.6) -- (-1.95,-8.4); 
\draw [dotted] (-2.1,-8.6) -- (-3.4,-9.9);
\draw (-3.55,-10.1) -- (-4.45,-10.9);

\draw (0,-8.6) -- (-0.95,-9.4); 
\draw [dotted] (-1.1,-9.6) -- (-2.4,-10.9);
\draw (-2.55,-11.1) -- (-3.45,-11.9);

\draw (2,-10.6) -- (1.05,-11.4); 
\draw [dotted] (0.9,-11.6) -- (-0.4,-12.9);
\draw (-0.55,-13.1) -- (-1.45,-13.9);

\draw (3,-11.6) -- (2.05,-12.4); 
\draw [dotted] (1.9,-12.6) -- (0.6,-13.9);
\draw (0.45,-14.1) -- (-0.45,-14.9);

\draw (-2.6,-2.5) -- (-4.9,-3.5); 
\draw (-1.6,-3.5) -- (-3.9,-4.5); 
\draw (0.4,-5.5) -- (-1.9,-6.5); 
\draw (1.4,-6.5) -- (-0.9,-7.5); 
\draw (2.4,-7.5) -- (0.1,-8.5); 
\draw (4.4,-9.5) -- (2.1,-10.5); 
\draw (5.4,-10.5) -- (3.1,-11.5); 

\draw (-3.6,-3.5) -- (-5.9,-4.5); 
\draw (-2.6,-4.5) -- (-4.9,-5.5); 
\draw (-0.6,-6.5) -- (-2.9,-7.5); 
\draw (0.4,-7.5) -- (-1.9,-8.5); 
\draw (1.4,-8.5) -- (-0.9,-9.5); 
\draw (3.4,-10.5) -- (1.1,-11.5); 
\draw (4.4,-11.5) -- (2.1,-12.5); 

\draw (0,-0.1) -- (-0.2,-0.1); 
\draw (0,-0.1) -- (-0.05,-0.3); 

\draw (1,-1.1) -- (0.8,-1.1); 
\draw (1,-1.1) -- (0.95,-1.3); 
\draw (3,-3.1) -- (2.8,-3.1); 
\draw (3,-3.1) -- (2.95,-3.3); 
\draw (4,-4.1) -- (3.8,-4.1); 
\draw (4,-4.1) -- (3.95,-4.3); 
\draw (5,-5.1) -- (4.8,-5.1); 
\draw (5,-5.1) -- (4.95,-5.3); 
\draw (7,-7.1) -- (6.8,-7.1); 
\draw (7,-7.1) -- (6.95,-7.3); 
\draw (8,-8.1) -- (7.8,-8.1); 
\draw (8,-8.1) -- (7.95,-8.3); 

\draw (-2.5,-2.6) -- (-2.7,-2.6); 
\draw (-2.5,-2.6) -- (-2.55,-2.8); 
\draw (-1.5,-3.6) -- (-1.7,-3.6); 
\draw (-1.5,-3.6) -- (-1.55,-3.8); 
\draw (0.5,-5.6) -- (0.3,-5.6); 
\draw (0.5,-5.6) -- (0.45,-5.8); 
\draw (1.5,-6.6) -- (1.3,-6.6); 
\draw (1.5,-6.6) -- (1.45,-6.8); 
\draw (2.5,-7.6) -- (2.3,-7.6); 
\draw (2.5,-7.6) -- (2.45,-7.8); 
\draw (4.5,-9.6) -- (4.3,-9.6); 
\draw (4.5,-9.6) -- (4.45,-9.8); 
\draw (5.5,-10.6) -- (5.3,-10.6); 
\draw (5.5,-10.6) -- (5.45,-10.8); 

\draw (-5,-3.6) -- (-5.2,-3.6); 
\draw (-5,-3.6) -- (-5.05,-3.8); 
\draw (-4,-4.6) -- (-4.2,-4.6); 
\draw (-4,-4.6) -- (-4.05,-4.8); 
\draw (-2,-6.6) -- (-2.2,-6.6); 
\draw (-2,-6.6) -- (-2.05,-6.8); 
\draw (-1,-7.6) -- (-1.2,-7.6); 
\draw (-1,-7.6) -- (-1.05,-7.8); 
\draw (0,-8.6) -- (-0.2,-8.6); 
\draw (0,-8.6) -- (-0.05,-8.8); 
\draw (2,-10.6) -- (1.8,-10.6); 
\draw (2,-10.6) -- (1.95,-10.8); 
\draw (3,-11.6) -- (2.8,-11.6); 
\draw (3,-11.6) -- (2.95,-11.8); 

\draw (-7.5,-6.1) -- (-7.7,-6.1); 
\draw (-7.5,-6.1) -- (-7.55,-6.3); 
\draw (-6.5,-7.1) -- (-6.7,-7.1); 
\draw (-6.5,-7.1) -- (-6.55,-7.3); 
\draw (-4.5,-9.1) -- (-4.7,-9.1); 
\draw (-4.5,-9.1) -- (-4.55,-9.3); 
\draw (-3.5,-10.1) -- (-3.7,-10.1); 
\draw (-3.5,-10.1) -- (-3.55,-10.3); 
\draw (-2.5,-11.1) -- (-2.7,-11.1); 
\draw (-2.5,-11.1) -- (-2.55,-11.3); 
\draw (-0.5,-13.1) -- (-0.7,-13.1); 
\draw (-0.5,-13.1) -- (-0.55,-13.3); 
\draw (0.5,-14.1) -- (0.3,-14.1); 
\draw (0.5,-14.1) -- (0.45,-14.3); 

\draw (-2.6,-2.5) -- (-2.8,-2.5); 
\draw (-2.6,-2.5) -- (-2.75,-2.65); 
\draw (-1.6,-3.5) -- (-1.8,-3.5); 
\draw (-1.6,-3.5) -- (-1.75,-3.65); 
\draw (0.4,-5.5) -- (0.2,-5.5); 
\draw (0.4,-5.5) -- (0.25,-5.65); 
\draw (1.4,-6.5) -- (1.2,-6.5); 
\draw (1.4,-6.5) -- (1.25,-6.65); 
\draw (2.4,-7.5) -- (2.2,-7.5); 
\draw (2.4,-7.5) -- (2.25,-7.65); 
\draw (4.4,-9.5) -- (4.2,-9.5); 
\draw (4.4,-9.5) -- (4.25,-9.65); 
\draw (5.4,-10.5) -- (5.2,-10.5); 
\draw (5.4,-10.5) -- (5.25,-10.65); 

\draw (-3.6,-3.5) -- (-3.8,-3.5); 
\draw (-3.6,-3.5) -- (-3.75,-3.65); 
\draw (-2.6,-4.5) -- (-2.8,-4.5); 
\draw (-2.6,-4.5) -- (-2.75,-4.65); 
\draw (-0.6,-6.5) -- (-0.8,-6.5); 
\draw (-0.6,-6.5) -- (-0.75,-6.65); 
\draw (0.4,-7.5) -- (0.2,-7.5); 
\draw (0.4,-7.5) -- (0.25,-7.65); 
\draw (1.4,-8.5) -- (1.2,-8.5); 
\draw (1.4,-8.5) -- (1.25,-8.65); 
\draw (3.4,-10.5) -- (3.2,-10.5); 
\draw (3.4,-10.5) -- (3.25,-10.65); 
\draw (4.4,-11.5) -- (4.2,-11.5); 
\draw (4.4,-11.5) -- (4.25,-11.65); 

\draw (0,-0.1) -- (0.2,-0.1); 
\draw (0,-0.1) -- (0.05,-0.3); 
\draw (3,-3.1) -- (3.2,-3.1); 
\draw (3,-3.1) -- (3.05,-3.3); 
\draw (4,-4.1) -- (4.2,-4.1); 
\draw (4,-4.1) -- (4.05,-4.3); 
\draw (7,-7.1) -- (7.2,-7.1); 
\draw (7,-7.1) -- (7.05,-7.3); 

\draw (-1,-1.1) -- (-0.8,-1.1); 
\draw (-1,-1.1) -- (-0.95,-1.3); 
\draw (2,-4.1) -- (2.2,-4.1); 
\draw (2,-4.1) -- (2.05,-4.3); 
\draw (3,-5.1) -- (3.2,-5.1); 
\draw (3,-5.1) -- (3.05,-5.3); 
\draw (6,-8.1) -- (6.2,-8.1); 
\draw (6,-8.1) -- (6.05,-8.3); 

\draw (-2.5,-2.6) -- (-2.3,-2.6); 
\draw (-2.5,-2.6) -- (-2.45,-2.8); 
\draw (0.5,-5.6) -- (0.7,-5.6); 
\draw (0.5,-5.6) -- (0.55,-5.8); 
\draw (1.5,-6.6) -- (1.7,-6.6); 
\draw (1.5,-6.6) -- (1.55,-6.8); 
\draw (4.5,-9.6) -- (4.7,-9.6); 
\draw (4.5,-9.6) -- (4.55,-9.8); 

\draw (-3.5,-3.6) -- (-3.3,-3.6); 
\draw (-3.5,-3.6) -- (-3.45,-3.8); 
\draw (-0.5,-6.6) -- (-0.3,-6.6); 
\draw (-0.5,-6.6) -- (-0.45,-6.8); 
\draw (0.5,-7.6) -- (0.7,-7.6); 
\draw (0.5,-7.6) -- (0.55,-7.8); 
\draw (3.5,-10.6) -- (3.7,-10.6); 
\draw (3.5,-10.6) -- (3.55,-10.8); 

\draw (-5,-3.6) -- (-4.8,-3.6); 
\draw (-5,-3.6) -- (-4.95,-3.8); 
\draw (-2,-6.6) -- (-1.8,-6.6); 
\draw (-2,-6.6) -- (-1.95,-6.8); 
\draw (-1,-7.6) -- (-0.8,-7.6); 
\draw (-1,-7.6) -- (-0.95,-7.8); 
\draw (2,-10.6) -- (2.2,-10.6); 
\draw (2,-10.6) -- (2.05,-10.8); 

\draw (-6,-4.6) -- (-5.8,-4.6); 
\draw (-6,-4.6) -- (-5.95,-4.8); 
\draw (-3,-7.6) -- (-2.8,-7.6); 
\draw (-3,-7.6) -- (-2.95,-7.8); 
\draw (-2,-8.6) -- (-1.8,-8.6); 
\draw (-2,-8.6) -- (-1.95,-8.8); 
\draw (1,-11.6) -- (1.2,-11.6); 
\draw (1,-11.6) -- (1.05,-11.8); 

\draw (-7.5,-6.1) -- (-7.3,-6.1); 
\draw (-7.5,-6.1) -- (-7.45,-6.3); 
\draw (-4.5,-9.1) -- (-4.3,-9.1); 
\draw (-4.5,-9.1) -- (-4.45,-9.3); 
\draw (-3.5,-10.1) -- (-3.3,-10.1); 
\draw (-3.5,-10.1) -- (-3.45,-10.3); 
\draw (-0.5,-13.1) -- (-0.3,-13.1); 
\draw (-0.5,-13.1) -- (-0.45,-13.3); 

\draw (-8.5,-7.1) -- (-8.3,-7.1); 
\draw (-8.5,-7.1) -- (-8.45,-7.3); 
\draw (-5.5,-10.1) -- (-5.3,-10.1); 
\draw (-5.5,-10.1) -- (-5.45,-10.3); 
\draw (-4.5,-11.1) -- (-4.3,-11.1); 
\draw (-4.5,-11.1) -- (-4.45,-11.3); 
\draw (-1.5,-14.1) -- (-1.3,-14.1); 
\draw (-1.5,-14.1) -- (-1.45,-14.3); 

\end{tikzpicture} 
\caption{Diagram of the weights of the $\frak{k}$-module $\frak{p}$ for $\frak{g}_0 = \frak{so}(2p,2l-2p+1)$}\label{diagram}
\end{center}
\end{figure} 

In the Figure \ref{diagram}, the vertices represent the roots in $\Delta_n$. Two roots $\beta, \gamma \in \Delta_n$ are joined by an edge with an arrow in the direction of 
$\gamma$ if $\gamma = \beta + \phi$ for some $\phi \in \Phi_\frak{k}$. In this case, the simple root $\phi$ is given on one side of the edge. From this diagram, we can 
clearly see that the lowest weight of the $\frak{l}'_0$-submodule of $\frak{p}$ with highest weight $\lambda$ is $\lambda - 2(\phi_{p+1}+ \cdots + 2\phi_l) - (\phi_1 + 
\cdots +\phi_{p-2}) - \epsilon = -(\phi_p + 2\phi_{p+1} +\cdots +2\phi_l) = \nu'.$ Also $P' \cap \Delta_n = \{ \beta \in P \cap \Delta_n : \beta \ge \phi_{p-1}+\phi_p\} \cup 
\{ \beta \in -P \cap \Delta_n : \beta \ge \nu'\} = \{\beta \in \Delta_n : n_{\phi_{p-1}}(\beta) = 1 \}\cup \{\beta \in -P \cap \Delta_n : n_{\phi_{p-1}}(\beta)=0 \}.$ 

%
%
%
\noindent 
\section{Discrete series representations} 

We will continue with notations from the previous section. Also assume that $G$ be a connected simple Lie group with finite centre and the Lie 
algebra of $G$ is $\frak{g}_0,$ and $K$ be the connected Lie subgroup of $G$ with Lie algebra $\frak{k}_0.$ Since rank$(G)=$ rank$(K), G$ admits 
discrete series representations. A non-singular linear function $\gamma$ on $i\frak{t}_0$ relative to $\Delta$ defines uniquely a 
positive root system $P_\gamma$ of $\Delta.$ Define $\rho_\frak{g} = \frac{1}{2}\sum_{\alpha \in P_\gamma}\alpha, \rho_\frak{k} = \frac{1}{2} 
\sum_{\alpha \in P_\gamma \cap \Delta_\frak{k}}\alpha.$ If $\gamma +\rho_\frak{g}$ is 
analytically integral(that is, $\gamma +\rho_\frak{g}$ is the differential of a Lie group homomorphism on the Cartan subgroup of $G$ corresponding to $\frak{t}_0$), then there exists a 
discrete series representation $\pi_\gamma$ with infinitesimal character $\chi_\gamma;$ the associated $(\frak{g}, K)$-module $\pi_{\gamma,K}$ contains an 
irreducible $K$-submodule with highest weight $\Gamma=\gamma+\rho_\frak{g}-2\rho_\frak{k};$
it occurs with multiplicity one in $\pi_{\gamma, K}.$ Any other irreducible $K$-module that occurs in $\pi_{\gamma, K}$ has highest weight 
of the form $\Gamma + \sum_{\alpha \in P_\gamma} n_\alpha \alpha,$ with $n_\alpha$ a non-negative integer; and 
two such representations $\pi_\gamma,$ and $\pi_{\gamma'}$ are unitarily equivalent {\it iff} $\gamma = w\gamma'$ for some $w \in W_\frak{k}$ \cite[Th. 9.20, Ch. IX]{knapp}. 
This $\gamma$ is called the {\it Harish-Chandra parameter}, $\Gamma$ is called the {\it Blattner parameter} of the discrete series representation $\pi_\gamma,$ 
and the positive root system $P_\gamma$ is called the {\it Harish-Chandra root order} corresponding to $\gamma.$
Upto unitary equivalence, these are all discrete series representations of $G$ \cite[Th. 12.21, Ch. XII]{knapp}. The condition $\gamma = w\gamma'$ for some $w \in W_\frak{k}$ 
implies that $P_\gamma \cap \Delta_\frak{k} = w(P_{\gamma'} \cap \Delta_\frak{k}).$ Hence to get non-equivalent discrete series 
representations, we may assume that the Harish-Chandra root order $P_\gamma$ corresponding to $\gamma$ contains $P_\frak{k}$ 
so that $P_\gamma \cap \Delta_\frak{k} = P_\frak{k}.$ Here the infinitesimal character $\chi_\gamma$ is the character of the Verma module of 
$\frak{g}$ with highest weight $\gamma-\rho_\frak{g}.$ We have $\chi_\gamma = \chi_{\gamma'}$ {\it iff} $\gamma' = w\gamma$ for some $w \in W_\frak{g}$ \cite[Th. 5.62, Ch. V]{knappb}. 
Thus the set of all mutually unitary inequivalent discrete series representations with infinitesimal character $\chi_\gamma$ is 
$\{ \pi_{w\gamma} : w \in W_\frak{g}, P_\frak{k} \subset P_{w\gamma} \}.$ Hence the number of unitary equivalence classes of discrete series representations with 
infinitesimal character $\chi_\gamma$ is $|W_\frak{g}|/|W_\frak{k}|.$ 

Assume that $G/K$ is Hermitian symmetric that is, $\frak{k}_0$ has non-zero centre, 
$P$ be a Borel-de Siebenthal positive root system containing $P_\frak{k},$ and $\rho_\frak{g} = \frac{1}{2}\sum_{\alpha \in P}\alpha.$ 
Then the symmetric space $G/K$ can be identified 
as an open subset of $\frak{l}_1 = \sum_{\beta \in P \cap \Delta_n} \frak{g}^\beta,$ and the complex structure of $G/K$ comes from the complex structure of the 
complex vector space $\frak{l}_1.$ See \cite[Ch. VIII]{helgason}. Let $\Lambda$ be a linear function on $i\frak{t}_0$ such that $\Lambda$ is analytically integral, dominant relative to 
$P_\frak{k},$ and $\langle \Lambda +\rho_\frak{g} , \beta \rangle < 0,$ for all $\beta \in P \cap \Delta_n.$ 
Then there exists a discrete series representation $\pi_\Lambda$ whose Harish-Chandra parameter is $\Lambda +\rho_\frak{g},$ 
the Harish-Chandra root order is $P_\frak{k} \cup (-P \cap \Delta_n),$ and the Blattner parameter is $\Lambda.$ This representation $\pi_\Lambda$ is defined by Harish-Chandra and it is called 
holomorphic discrete series representation (\cite{repka}). This discrete series representation $\pi_\Lambda$ is called holomorphic because its representation space is the set of all 
holomorphic maps from $G/K$ to $V_\Lambda$ (the representation space of the finite-dimensional irreducible unitary representation of $K$ with highest weight $\Lambda$) with finite 
$L^2$-norm (relative to the left $G$-invariant measure on $G/K$). Similarly we can define an anti-holomorphic discrete series representation $\pi_{\Lambda'},$ where $\Lambda' \in 
(i\frak{t}_0)^*$ is analytically integral, dominant relative to $-P_\frak{k},$ and $\langle \Lambda' - \rho_\frak{g} , \beta \rangle < 0,$ for all $\beta \in -P \cap \Delta_n$ (\cite{repka}). If the 
symmetric space $G/K$ is identified as an open subset of $\frak{l}_{-1} = \sum_{\beta \in -P \cap \Delta_n} \frak{g}^\beta$ so that $G/K$ is equipped with the conjugate complex structure, 
then the representation space of $\pi_{\Lambda'}$ is the set of all 
holomorphic maps from $G/K$ to $V_{\Lambda'}$ (the representation space of the finite-dimensional irreducible unitary representation of $K$ with lowest weight $\Lambda$ relative to $P_\frak{k}$) 
with finite $L^2$-norm. Note that $\pi_{\Lambda'}$ is unitarily equivalent to $\pi_{w_\frak{k}^0(\Lambda')}$ ($w_\frak{k}^0 \in W_\frak{k}$ is the longest element relative to $P_\frak{k}$), and the 
Harish-Chandra root order of $\pi_{w_\frak{k}^0(\Lambda')}$ is $P.$ Now the Borel-de Siebenthal positive root system containing $P_\frak{k}$ are $P_\frak{k} \cup (-P \cap \Delta_n),$ and $P.$ 
The corresponding discrete series representations with infinitesimal character $\chi_{\Lambda +\rho_\frak{g}}$ are $\pi_\Lambda,$ which is holomorphic, and $\pi_{w^0(\Lambda)}$ ($w^0 = w_\frak{k}^0
w_\frak{g}^0, w_\frak{g}^0 \in W_\frak{g}$ is the longest element relative to $P$), which is anti-holomorphic, respectively. 

Now assume that $G/K$ is not Hermitian symmetric that is, $\frak{k}_0$ is semisimple, $P$ be a Borel-de Siebenthal positive root system containing $P_\frak{k},$ and 
$\rho_\frak{g} = \frac{1}{2}\sum_{\alpha \in P}\alpha.$ Then \O rsted and Wolf \cite{ow} have defined Borel-de Siebenthal discrete series representations of $G$ 
analogous to holomorphic discrete series representations as follows: Let $\gamma \in (it_0)^*$ be such that $\gamma$ is analytically integral, dominant relative to the positive root system 
$P_0 = P \cap \Delta_0$ of $\Delta_0 = \Delta(\frak{l}_0,\frak{h}),$ and $\langle \gamma +\rho_\frak{g} , \beta \rangle < 0,$ for all $\beta \in P \setminus P_0.$ 
The Borel-de Siebenthal discrete series representation $\pi_{\gamma + \rho_\frak{g}}$ of $G$, is a discrete series representation for which the Harish-Chandra parameter is $\gamma +\rho_\frak{g}.$ 
Clearly the Harish-Chandra root order is $P_0 \cup -(P \setminus P_0),$ and the Blattner parameter is $\gamma +\sum_{\beta \in P, n_\nu(\beta)=2} \beta.$

\noindent
{\bf Proof of Theorem \ref{rep}:} 
If $G/K$ is not Hermitian symmetric, consider the Borel-de Siebenthal discrete series representation $\pi_{\gamma + \rho_\frak{g}}$ of $G.$ The Harish-Chandra root order 
of $\pi_{\gamma + \rho_\frak{g}}$ is $P_0 \cup -(P \setminus P_0),$ which is a Borel-de Siebenthal positive root system: \\ 
For $-(P \setminus P_\frak{k}) = \{-\lambda + \sum_{\phi \in \Phi_0} n_\phi \phi : n_\phi \in \mathbb{N} \cup \{0\} \},$ 
$-(P_\frak{k} \setminus P_0) = \{-\mu + \sum_{\phi \in \Phi_0} n_\phi \phi : n_\phi \in \mathbb{N} \cup \{0\} \}$ and $-\mu = -2\lambda + \sum_{\phi \in \Phi_0} m_\phi \phi$ 
for some $m_\phi \in \mathbb{N} \cup \{0\} (\phi \in \Phi).$ Thus the set of simple roots of $P_0 \cup -(P \setminus P_0)$ is $\Phi_0 \cup \{-\lambda\},$ which has exactly 
one non-compact root $-\lambda,$ and the coefficient of $-\lambda$ in the highest root $-\epsilon$ of $P_0 \cup -(P \setminus P_0),$ when expressed as the sum of 
roots in $\Phi_0 \cup \{-\lambda\},$ is $2.$ \\ 
Since $w'_0=w_{\frak{l}_0}^0w_\frak{k}^0 \in W_\frak{k}$ (recall that $w_{\frak{l}_0}^0 \in W_{\frak{l}_0}$ is the 
longest element relative to $P_0$), we have $\pi_{\gamma + \rho_\frak{g}}$ is unitarily equivalent to $\pi_{w'_0(\gamma + \rho_\frak{g})}.$ The Harish-Chandra root order 
of $\pi_{w'_0(\gamma + \rho_\frak{g})}$ is the Borel-de Siebenthal positive root system $w'_0(P_0 \cup -(P \setminus P_0))$ containing $P_\frak{k}.$ 
Thus the set of all mutually unitary inequivalent Borel-de Siebenthal discrete series representations with infinitesimal character $\chi_{\gamma + \rho_\frak{g}}$ is 
$\{ \pi_{ww'_0(\gamma +\rho_\frak{g})} : w \in W_\frak{g}, ww'_0(P_0 \cup -(P \setminus P_0)) \textrm{ is a Borel-de Siebenthal positive root system containing } P_\frak{k} \}.$ 
Hence the number of unitary equivalence classes of Borel-de Siebenthal discrete series representations with 
infinitesimal character $\chi_{\gamma+\rho_\frak{g}}$ is the number of Borel-de Siebenthal positive root systems of $\Delta$ containing $P_\frak{k},$ which is given 
in the statement of Theorem \ref{rep}.

\section*{acknowledgement}

The author acknowledges the financial support from the Department of Science and Technology (DST), Govt. of India under the Scheme 
"Fund for Improvement of S\&T Infrastructure (FIST)" [File No. SR/FST/MS-I/2019/41].


\begin{thebibliography}{99} 
\bibitem{bds} Borel, A.; de Siebenthal, J. Les sous-groupes ferm\'{e}s de rang maximum des groupes de Lie clos. Comment. Math. Helv. {\bf 23} (1949) 200--221. 
\bibitem{hc1} Harish-Chandra Representations of semisimple Lie groups. VI. integrable and square-integrable representations. Amer. J. Math. {\bf 78} (1956) 564--628.
\bibitem{helgason} Helgason, Sigurdur {\it Differential geometry, Lie groups, and symmetric spaces.} Corrected reprint of the 1978 original. Graduate Studies in Mathematics 
{\bf 34}, American Mathematical Society, 2001.
\bibitem{knappb} Knapp, A.W. {\it Lie Groups Beyond an Introduction.} Progress in Mathematics {\bf 140}, Birkh\"{a}user, 2002. Second Ed. 
\bibitem{knapp} Knapp, A. W. {\it Representation Theory of Semisimple Groups. An Overview Based on Examples.} Princeton Landmarks in Mathematics. Princeton University Press, 2001.
\bibitem{ow} \O rsted, Bent; Wolf, Joseph A. Geometry of the Borel-de Siebenthal discrete series. J. Lie Theory {\bf 20} (2010) 175--212.
\bibitem{repka} Repka, J. Tensor products of holomorphic discrete series representations. Can. J. Math. {\bf 31}(1979) 836--844. 
\end{thebibliography}
\end{document}